\documentclass[a4paper,11pt]{article}
\usepackage{amssymb}
\usepackage{amsfonts}
\usepackage{amsmath,tabu}
\usepackage{amsthm}
\usepackage{cite}

\renewcommand{\k}{\mathbf{k}}
\renewcommand{\j}{{\mathbf{j}}}
\newcommand{\s}{s}
\newcommand{\eps}{{\varepsilon}}

\newcommand{\R}{\mathbb{R}}

\newcommand{\N}{\mathbb{N}}

\newcommand{\cuad}{{\sqcap\kern-.68em\sqcup}}

\newcommand{\norm}[1]{\|#1\|}

\numberwithin{equation}{section}
\newtheorem{theorem}{Theorem}[section]
\newtheorem{definition}[theorem]{Definition}
\newtheorem{proposition}[theorem]{Proposition}

\newtheorem{lemma}[theorem]{Lemma}
\newtheorem{corollary}[theorem]{Corollary}
\newtheorem{remark}[theorem]{Remark}
\newcommand{\bremark}{\begin{remark} \em}
\newcommand{\eremark}{\end{remark} }

\newcommand{\cE}{{\mathcal E}}
\newcommand{\cF}{{\mathcal F}}

\newcommand{\cH}{{\mathbb H}}

\newcommand{\cM}{{\mathcal M}}

\newcommand{\cS}{{\mathcal S}}

\newcommand{\cV}{{\mathcal V}}
\newcommand{\cW}{{\mathcal W}}

\newcommand{\loglap}{L_{\text{\tiny $\Delta \,$}}\!}

\DeclareMathOperator{\dist}{\rm dist}
\DeclareMathOperator{\supp}{\rm supp \,}

\begin{document}
\begin{center}{\bf  \large  The Dirichlet Problem for the Logarithmic Laplacian}\medskip

 \bigskip

 {\small  Huyuan Chen\footnote{chenhuyuan@yeah.net} \qquad Tobias Weth\footnote{weth@math.uni-frankurt.de}
}
 \bigskip

{\small  $ ^1$Department of Mathematics, Jiangxi Normal University, Nanchang,\\
Jiangxi 330022, PR China\\[3mm]

$ ^2$Goethe-Universit\"{a}t Frankfurt, Institut f\"{u}r Mathematik, Robert-Mayer-Str. 10,\\
D-60629 Frankfurt, Germany
 } \\[6mm]

\begin{abstract}
In this paper,  we study the logarithmic Laplacian operator  $\loglap$, which is a singular integral operator with symbol $2\log |\zeta|$.
We show that this operator has the integral representation
$$\loglap  u(x) =
c_{N} \int_{\R^N  } \frac{ u(x)1_{B_1(x)}(y)-u(y)}{|x-y|^{N} } dy + \rho_N u(x)
$$
with $c_N = \pi^{- \frac{N}{2}}  \Gamma(\frac{N}{2})$ and $\rho_N=2 \log 2 + \psi(\frac{N}{2}) -\gamma$, where $\Gamma$ is the Gamma function, $\psi = \frac{\Gamma'}{\Gamma}$ is the Digamma function and $\gamma= -\Gamma'(1)$ is the Euler Mascheroni constant. This operator arises as formal derivative $\partial_\s \Big|_{\s=0} (-\Delta)^\s$ of fractional Laplacians at $\s= 0$. We develop the functional analytic framework for Dirichlet problems involving the logarithmic Laplacian on bounded domains and use it to characterize the asymptotics of principal Dirichlet eigenvalues and eigenfunctions of $(-\Delta)^\s$ as $\s \to 0$. As a byproduct, we then derive a Faber-Krahn type inequality for the principal Dirichlet eigenvalue of $\loglap$. Using this inequality, we also establish conditions on domains giving rise to the maximum principle in weak and strong forms. This allows us to also derive regularity up to the boundary of solutions to corresponding Poisson problems. \end{abstract}

\end{center}
 \vspace{1mm}


\vspace{1mm}

\setcounter{equation}{0}
\section{Introduction and main results}

In recent years, there has been a renewed and increasing
interest in the study of boundary value problems involving linear and nonlinear integro-differential operators.
This growing interest is fueled both by important applications and seminal advances in the
understanding of nonlocal phenomena from a PDE point of view, see e.g.
\cite{CS0,CS1,CS2,CFQ,EGE,FQ2,RS,RS1,S} and the references therein.
Among nonlocal operators of positive differential order, fractional powers of the Laplacian are receiving the most attention.
Recall that, for $\s\in(0,1)$, the fractional Laplacian of a function $u \in C^\infty_c(\R^N)$ is defined by
\begin{equation}
  \label{eq:Fourier-representation}
\mathcal{F}((-\Delta)^\s u)(\xi) = |\xi|^{2\s}\widehat u (\xi)\qquad \text{for all $\xi \in \R^N$},
\end{equation}
where both $\mathcal{F}$ and $\widehat \cdot$ denote the Fourier transform. Equivalently, we can write $(-\Delta)^\s$ as a singular integral operator, i.e.
 \begin{equation}\label{fl 1}
 (-\Delta)^\s  u(x)=c_{N,\s} \lim_{\epsilon\to0^+} \int_{\R^N\setminus B_\epsilon(x) }\frac{u(x)-
u(y)}{|x-y|^{N+2\s}}  dy = c_{N,\s} \lim_{\epsilon\to0^+} \int_{\R^N\setminus B_\epsilon(0) }\frac{u(x)-
u(x+z)}{|z|^{N+2\s}}  dz,
\end{equation}
where $c_{N,\s}=2^{2\s}\pi^{-\frac N2}\s\frac{\Gamma(\frac{N+2\s}2)}{\Gamma(1-\s)}$ and $\Gamma$  is the  Gamma function, see e.g. \cite{RS1}.

It is well known that the fractional Laplacian has the following limiting
properties when $\s$ approaches the values zero and $1$:
\begin{equation}
  \label{eq:limit-behaviour}
\lim_{\s\to1^-}(-\Delta)^\s u(x)=-\Delta u(x)
\quad \text{and}\quad \lim_{\s\to0^+}  (-\Delta)^\s  u(x) = u(x)\qquad \text{for $u\in C^2_c(\R^N)$,}
\end{equation}
see e.g. \cite{EGE}. On the other hand, one might guess that the function $z \mapsto |z|^{-N}$ also appears in
a suitably renormalized limit of the integral kernels $z \mapsto \frac{c_{N,\s}}{|z|^{N+2\s}}$ appearing in (\ref{fl 1}) as $\s \to 0$. However, due to lack of integrability, the zero order kernel $|\cdot|^{-N}$ has to be cut off at infinity in order to give rise to a singular integral operator defined analogously as in (\ref{fl 1}). Integral operators given by kernels with a singularity of the order $-N$ have received growing interest recently, as they give rise to interesting limiting regularity properties and Harnack inequalities without scaling invariance, see e.g. \cite{KM}.

As we shall see in the present paper, operators of such limiting order appear also within a first order expansion of the second limit in (\ref{eq:limit-behaviour}). More precisely, we find that for $u \in C^2_c(\R^N)$ and  $x \in \R^N$,
$$
(-\Delta)^\s  u(x) = u(x) + s \loglap u (x) + o(\s) \quad \text{as\ \, $\s\to0^+$,}
$$
where, formally, the operator $\loglap:= \frac{d}{d\s}\Big|_{\s=0} (-\Delta)^\s$ is given as a {\em logarithmic Laplacian}. Indeed, this operator has a logarithmic symbol, since, at least formally, (\ref{eq:Fourier-representation}) gives rise to the representation
$$\mathcal{F}(\loglap u)(\xi) = (2 \log |\xi|)\, \widehat u (\xi),\qquad \forall\,\xi \in \R^N.
$$
The purpose of the present paper is to study qualitative and functional analytic properties of the extremal nonlocal operator  $\loglap $  and related Dirichlet problems. An important motivation for this study is the fact that, as we shall observe, the operator $\loglap$ appears in the asymptotic description of the first Dirichlet eigenvalue of $(-\Delta)^s$ in bounded domains for small $s$ and of corresponding eigenfunctions.

In the following, we present the main results of the present paper. Our first result provides an integral representation of the logarithmic Laplacian.

\begin{theorem}\label{teo-representation}
Let $u \in C^\beta_c(\R^N)$ for some $\beta>0$. Then we
have
\begin{align}
 \loglap  u(x)&:= \frac{d}{d\s}\Big|_{\s=0} [(-\Delta)^\s u](x)= c_{N} \int_{\R^N  } \frac{ u(x)1_{B_1(x)}(y)-u(y)}{|x-y|^{N} } dy + \rho_N u(x)   \label{representation-main}\\
   &=  c_{N}  \int_{B_1(x)  } \frac{ u(x) -u(y)}{|x-y|^{N} } dy-c_N\int_{\R^N\setminus B_1(x)}\frac{u(y)}{|x-y|^{N} } dy + \rho_N u(x) \quad \text{for \,$x \in \R^N$,}\nonumber
\end{align}
where
\begin{equation}
  \label{eq:def-c-N}
c_N:= \pi^{-  N/2}  \Gamma( N/2) = \frac{2}{|S^{N-1}|}, \qquad \rho_N:=2 \log 2 + \psi(\frac{N}{2}) -\gamma
\end{equation}
and $\gamma= -\Gamma'(1)$ is the Euler Mascheroni constant. Here $\psi = \frac{\Gamma'}{\Gamma}$ is the Digamma function. Moreover,
\begin{enumerate}
\item[(i)] for $1 < p \le \infty$, we have $\loglap  u \in L^p(\R^N)$ and $\frac{(-\Delta)^\s u- u}{\s} \to \loglap  u$ in $L^p(\R^N)$ as $\s \to 0^+$;
\item[(ii)] $\mathcal{F}(\loglap u)(\xi) = (2 \log |\xi|)\, \widehat u (\xi)$ 
 \, for a.e. $\xi \in \R^N$.
\end{enumerate}
\end{theorem}

\begin{remark}{\rm 
We note that
$\rho_N= -2 \gamma + \sum \limits_{k=1}^{(N-1)/2} \frac{2}{2k-1}$ for $N$ odd and
$\rho_N = 2 \bigl(\log 2 -\gamma\bigr)+ \sum \limits_{k=1}^{(N-2)/2}\frac{1}{k}$ for $N$ even. We also note that, in distributional sense, the integral representiation (\ref{representation-main}) coincides, up to a constant, with the convolution with the inverse Fourier transform of the symbol $2 \log |\cdot|$, see e.g. \cite[Chapter II.3]{gelfand-shilov}.
}
\end{remark}

The formula (\ref{representation-main}) allows to define $\loglap u$ for a fairly large class of functions $u$. For this we recall that, for $s \in \R$, the space $L^1_s(\R^N)$ denotes the space of locally integrable functions $u: \R^N \to \R$ such that
$$
\|u\|_{L^1_s}:= \int_{\R^N} \frac{|u(x)|}{(1+|x|)^{N+2s}}\,dx < +\infty.
$$

\begin{proposition}
\label{dini-continuity-well-defined-theorem}
Let $u \in L^1_0(\R^N)$. If $u$ is Dini continuous at some $x \in \R^N$, then $[\loglap u](x)$ is well-defined by the formula (\ref{representation-main}). Moreover, if $u$ is uniformly Dini continuous in some open subset $\Omega \subset \R^N$, then $\loglap u$ is continuous in $\Omega$.
\end{proposition}

For the definition of (uniform) Dini continuity, see Section~\ref{sec:basic-properties} below. Our next aim is to study the eigenvalue problem
\begin{equation}\label{eq 3.1-main-results}
\left\{ \arraycolsep=1pt
\begin{array}{lll}
 \loglap u=\lambda  u\quad \  {\rm in}\quad   \Omega,\\[2mm]
 \phantom{ \loglap   }
  u=0\quad \ {\rm{in}}\  \quad \R^N\setminus \Omega
\end{array}
\right.
\end{equation}
in a bounded domain $\Omega$. We consider corresponding eigenfunctions in weak sense. The corresponding functional analytic framework is given as follows.
Let $\cH(\Omega)$ denote the space of all measurable functions $u:\R^N\to \R$ with $u \equiv 0$ on $\R^N \setminus \Omega$ and
$$
\int \!\!\!\! \int_{\stackrel{x,y \in \R^N}{\text{\tiny $|x-y|\!\le\! 1$}}} \frac{(u(x)-u(y))^2}{|x-y|^N} dx dy <+\infty.
$$
We shall see that $\cH(\Omega)$ is a Hilbert space with inner product
$$
\mathcal{E}(u,w)=\frac{c_N}2 \int \!\!\int_{\stackrel{x,y \in \R^N}{\text{\tiny $|x-y|\!\le\! 1$}}} \frac{(u(x)-u(y))(w(x)-w(y))}{|x-y|^N}
dx dy
$$
and the induced norm $\norm{u}_{\cH(\Omega)}=\sqrt{\mathcal{E}(u,u)}$, where $c_N$ is given in (\ref{eq:def-c-N}). By \cite[Theorem 2.1]{CP}, the embedding $\cH(\Omega) \hookrightarrow L^2(\Omega)$ is compact. Here and in the following, we identify $L^2(\Omega)$ with the space of functions in $L^2(\R^N)$ with $u \equiv 0$ on $\R^N \setminus \Omega$. Moreover, the quadratic form associated with $\loglap$ is well-defined on $\cH(\Omega)$ by
$$
\cE_L: \cH(\Omega) \times \cH(\Omega) \to \R, \quad \cE_L(u,w)= \mathcal{E}(u,w) - c_N
\int \!\!\!\! \int_{\stackrel{x,y \in \R^N}{\text{\tiny $|x-y|\!\ge\! 1$}}} \frac{u(x)w(y)}{|x-y|^N} dx dy + \rho_N \int_{\R^N}  u w\,dx.
$$
A function $u \in \cH(\Omega)$ will then be called an eigenfunction of (\ref{eq 3.1-main-results}) corresponding to the eigenvalue
$\lambda$ if
$$
\cE_L(u,\phi) = \lambda \int_{\Omega}u\phi\,dx \qquad \text{for all $\phi \in \cH(\Omega)$.}
$$

\begin{theorem}\label{pr eigen 1}
Let $\Omega$ be a bounded domain  in $\R^N$.  Then problem (\ref{eq 3.1-main-results}) admits a sequence of eigenvalues
$$
\lambda_1^L(\Omega)<\lambda_2^L(\Omega)\le \cdots\le \lambda_k^L(\Omega)\le \lambda_{k+1}^L(\Omega)\le \cdots
$$
and corresponding eigenfunctions $\xi_k$, $k \in \N$ such that the following holds:
\begin{enumerate}
\item[(i)] $\lambda_{k}^L(\Omega)=\min \{\cE_L(u,u) \::\:  u\in \cH_k(\Omega)\::\: \norm{u}_{L^2(\Omega)}=1\}$, where
$$
\cH_1(\Omega):= \cH(\Omega)\quad \text{and}\quad  \cH_k(\Omega):=\{u\in\cH(\Omega)\::\: \text{$\int_{\Omega} u \xi_i \,dx =0$ for $i=1,\dots k-1$}\}\quad \text{for $k>1$.}
$$
\item[(ii)] $\{\xi_k\::\: k \in \N\}$ is an orthonormal basis of $L^2(\Omega)$.
\item[(iii)] $\xi_1$ is strictly positive in $\Omega$. Moreover, $\lambda_1^L(\Omega)$ is simple, i.e., if $u \in \cH(\Omega)$ satisfies (\ref{eq 3.1-main-results}) in weak sense with $\lambda = \lambda_1^L(\Omega)$, then $u=t\xi_1$ for some $t\in\R$.
\item[(iv)] $\lim \limits_{k \to \infty} \lambda_k^L(\Omega)=+\infty$.
\end{enumerate}
\end{theorem}

Our next theorem highlights the role of $\lambda_1^L(\Omega)$ and the corresponding eigenfunction $\xi_1$.

\begin{theorem}
\label{sec:funct-analyt-fram-approximation-main}
Let $\Omega$ be a bounded Lipschitz domain  in $\R^N$, and let $\lambda_1^s(\Omega)$ denote the first Dirichlet eigenvalue of $(-\Delta)^s$ on $\Omega$ for $s \in (0,1)$. Then we have
$$
\lambda_1^L(\Omega) = \frac{d}{ds}\Big|_{s=0} \lambda_1^s(\Omega).
$$
Moreover, if $u_s$ is the unique nonnegative $L^2$-normalized Dirichlet eigenfunction
of $(-\Delta)^s$ corresponding to $\lambda_1^s(\Omega)$, then we have
$$
u_s \to \xi_1 \qquad \text{in $L^2(\Omega)$,}
$$
where $\xi_1$ is the unique nonnegative $L^2$-normalized eigenfunction of $\loglap$ corresponding to $\lambda_1^L(\Omega)$.
\end{theorem}

Since $\cE_L$ contains competing nonlocal terms of different signs, it seems unclear how the quadratic form $\cE_L$ changes under Schwarz symmetrization or other types of rearrangements. Nevertheless, as we shall see in Section~\ref{sec:funct-analyt-fram}, Theorem~\ref{sec:funct-analyt-fram-approximation-main} allows to deduce the Faber-Krahn-inequality for the logarithmic Laplacian from the corresponding one for the fractional Laplacian due to Ba{\~n}uelos et al., see \cite[Theorem 5]{BLMH}.

\begin{corollary}(Faber-Krahn-inequality for the logarithmic Laplacian)
\label{sec:faber-Krahn-main}
Let $\rho>0$. Among all bounded Lipschitz domains $\Omega$ with $|\Omega| = \rho$, the ball $B=B_r(0)$ with $|B|=\rho$ minimizes $\lambda_1^L(\Omega)$.
\end{corollary}

Next, we wish to discuss the maximum principle for the operator $\loglap u$ on bounded domains $\Omega$. In order to use this maximum principle for existence and regularity results, it is important to consider corresponding inequalities in weak sense.
For this we let $\cV(\Omega)$ denote the space of all measurable functions $u \in L^1_0(\R^N)$ such that
$$
\int \!\!\!\! \int_{\stackrel{x,y \in \Omega}{\text{\tiny $|x-y|\!\le\! 1$}}} \frac{(u(x)-u(y))^2}{|x-y|^N} dx dy + \int_{\Omega}u^2(x)\,dx <+\infty.
$$
We shall see in Section~\ref{sec:maximum-principle} below that the quadratic form $\cE_L(u,\phi)$ is well defined for $u \in \cV(\Omega)$, $\phi \in C^\infty_c(\Omega)$. We may now define a weak notion of the property $\loglap u \ge 0$ in $\Omega$.

\begin{definition}
\label{maximum-principle-definition-main}
\begin{enumerate}
\item[i)] For a function $u \in \cV(\Omega)$, we say that $\loglap u \ge 0$ in $\Omega$ (in weak sense) if $\cE_L(u, \phi) \ge 0$ for all nonnegative $\phi \in C_c^\infty(\Omega)$.
\item[ii)] We say that $\loglap$ satisfies the maximum principle in $\Omega$ if for every $u \in \cV(\Omega)$ with
$\loglap u \ge 0$ in $\Omega$ and $u \ge 0$ in $\R^N \setminus \Omega$ we have $u \ge 0$ a.e. in $\R^N$.
\end{enumerate}
\end{definition}

\begin{theorem}
 \label{sec:maxim-princ-bound-characterization}
Let $\Omega \subset \R^N$ be a bounded Lipschitz domain. Then $\loglap$ satisfies the maximum principle on $\Omega$ if and only if $\lambda_1^L(\Omega)>0$.
\end{theorem}

We briefly comment on the proof of this theorem. The basic idea is to test the inequality $\loglap u \ge 0$ with the function $u^-$. Indeed we shall see that $u^- \in \cH(\Omega)$ with $\cE_{L}(u^-,u^-)\le 0$ whenever $u \in \cV(\Omega)$ satisfies $\loglap u \ge 0$ in $\Omega$ and $u \ge 0$ in $\R^N \setminus \Omega$, see Proposition~\ref{sec:maxim-princ-bound-u-} below. Since our definition of the inequality $\loglap u \ge 0$ is merely based on testing with functions in $C^\infty_c(\Omega)$, it is also necessary to prove that this space is dense in $\cH(\Omega)$. We shall do this for bounded Lipschitz domains in Theorem~\ref{density} below. The following corollary follows from a combination of Corollary~\ref{sec:faber-Krahn-main} with Theorem~\ref{sec:maxim-princ-bound-characterization} and some further estimates.

\begin{corollary}
\label{sec:main-result-max-principle}
Let $\Omega \subset \R^N$ be a bounded Lipschitz domain. Then $\loglap$ satisfies the maximum principle in $\Omega$ if \underline{one} of the following conditions are satisfied:
\begin{enumerate}
\item[(i)] $h_\Omega + \rho_N \ge 0$ on $\Omega$, where
\begin{equation}
  \label{eq:def-h-Omega}
h_\Omega(x)=  c_N \Bigl( \int_{B_1(x)\setminus \Omega} \frac{1}{|x-y|^N}dy - \int_{\Omega \setminus B_1(x)} \frac{1}{|x-y|^N}dy\Bigr)
\end{equation}
and the constants $c_N$ and $\rho_N$ are given in Theorem~\ref{teo-representation};
\item[(ii)] $|\Omega| \le 2^N \exp \bigl(\frac{N}{2} \bigl(\psi(\frac{N}{2})-\gamma\bigr)\bigr)|B_1(0)|$.
\end{enumerate}

\end{corollary}

We remark that, by Corollary~\ref{sec:faber-Krahn-main} and Theorem~\ref{sec:maxim-princ-bound-characterization}, it suffices to consider the values of $\lambda_1^L(B_r(0))$ for $r>0$. In particular, we shall deduce Corollary \ref{sec:main-result-max-principle}(ii) from the fact that $\lambda_1^L(B_r(0)) >0$ if \mbox{$r\le r_N:= 2 \exp \bigl(\frac{1}{2}\bigl(\psi(\frac{N}{2})-\gamma \bigr)\bigr)$,} see Lemma~\ref{sec:maxim-princ-h-lower-bounds} below. It is instructive to compare this estimate with Beckner's logarithmic estimate of uncertainty (see \cite[Theorem 1]{B}), which implies that
\begin{equation}
  \label{eq:beckner-inequality}
\frac{\cE_L (u,u)}{2} =
 \int_{\R^N} |\hat u|^2 (\zeta) \log |\zeta| d\zeta \ge \int_{\R^N} \Bigl[\psi(\frac{N}{4})- \log (\pi |x|)\Bigr]u^2(x) dx
\end{equation}
for Schwarz functions $u \in \cS(\R^N)$. Clearly, the RHS of (\ref{eq:beckner-inequality}) is nonnegative if $\supp u \subset B_{r_{N,B}}$ with $r_{N,B} := \frac{e^{\psi (\frac N4)}}{\pi}$, and thus $\lambda_1^L(B_r(0)) \ge 0$ for $r < r_{N,B}$. The values of $r_N$ and $r_{N,B}$ can be computed explicitely for given $N$. In particular, $r_{N,B}$ is much smaller than $r_N$ in low dimensions, i.e. for $N \le 4$:
\vspace{-0.5cm}

\begin{center}
\renewcommand{\arraystretch}{1.6}
    \begin{tabular}{ l | l | l | l | l |}
    $N$ & 1 & 2 & 3 & 4\\ \hline
    $r_{N,B}$ & $\frac{e^{- \gamma -\frac{\pi}{2}}}{8 \pi} $& $\frac{e^{-\gamma }}{4 \pi}$&$\frac{e^{- \gamma +\frac{\pi}{2}}}{8 \pi}$&$\frac{e^{-\gamma}}{\pi}$
\\ \hline
$r_N$& $e^{-\gamma}$& $2 e^{-\gamma}$& $e^{1-\gamma}$& $2 e^{\frac{1}{2}-\gamma}$\\
\end{tabular}
\renewcommand{\arraystretch}{1}
\end{center}
Consequently, Corollary~\ref{sec:main-result-max-principle}(ii) cannot be deduced from (\ref{eq:beckner-inequality}) even in the case when $\Omega$ is a ball.
On the other hand, since $\psi(t) \sim \log t$ as $t \to \infty$, we have
$r_N < r_{N,B}$ for large $N$.

 Corollary~\ref{sec:main-result-max-principle}(ii) shows that $\loglap$ satisfies the maximum principle on domains $\Omega$ which are not too big in terms of Lebesgue measure. If, on the other hand, $\Omega$ has a large inradius, then $\loglap$ does not satisfy the maximum principle on $\Omega$. Indeed, we have the following result.

\begin{corollary}
\label{non-max-principle-simple}
Let $\Omega \subset \R^N$ be a bounded Lipschitz domain, and let $\lambda_1^1(\Omega)$ denote the first eigenvalue of the Dirichlet Laplacian $-\Delta$ on $\Omega$. Then we have
$\lambda_1^L(\Omega) \le \log \lambda_1^1(\Omega)$. Moreover, if $\lambda_1^1(\Omega) \le 1$, then $\loglap$ does not satisfy the maximum principle on $\Omega$.
\end{corollary}

This result follows from Theorems~\ref{sec:funct-analyt-fram-approximation-main},~\ref{sec:maxim-princ-bound-characterization} and an inequality in \cite{musina-nazarov} which relates $\lambda_1^1(\Omega)$ to the first Dirichlet eigenvalue $\lambda_1^s(\Omega)$ of $(-\Delta)^s$ for $s \in (0,1)$, see Section~\ref{sec:maximum-principle} below. 

Finally, we consider the Poisson problem
\begin{equation}\label{eq 3.1-poisson-main}
\left\{ \arraycolsep=1pt
\begin{array}{lll}
 \loglap u=f\quad \  {\rm in}\quad   \Omega,\\[2mm]
 \phantom{ \loglap   }
  u=0\quad \ {\rm{in}}\  \quad \R^N\setminus \Omega,
\end{array}
\right.
\end{equation}
with given $f \in L^2(\Omega)$ in weak sense. The appropriate framework is again given by the Hilbert space $\cH(\Omega)$ defined above. We say that $u \in \cH(\Omega)$ is a weak solution of (\ref{eq 3.1-poisson-main}) if
$$
\cE_L(u,\phi)= \int_{\Omega}f \phi\,dx \qquad \text{for every $\phi \in \cH(\Omega)$.}
$$

A standard application of the Riesz representation theorem shows that, if $\Omega$ is a bounded domain with $\lambda^L_1(\Omega)>0$, then (\ref{eq 3.1-poisson-main}) admits a weak solution $u \in \cH(\Omega)$ for every $f \in L^2(\Omega)$. Our final main result deals with the interior and boundary regularity of weak solutions of (\ref{eq 3.1-poisson-main}) in the case when
 $f \in L^\infty(\Omega)$.

\begin{theorem}\label{pr 2.1-main}
Let $\Omega \subset \R^N$ be a Lipschitz domain which satisfies a uniform exterior sphere condition, 
 $f \in L^\infty(\Omega)$ and  $u \in \cH(\Omega) \cap L^\infty(\R^N)$ be a weak solution of (\ref{eq 3.1-poisson-main}).
 Then $u \in C(\overline \Omega)$ and
\begin{equation}\label{boundary-decay-main}
|u(x)|= O\bigl(\log^{-\tau} \frac1{\rho(x)}\bigr) \qquad \text{for every $\tau \in (0,\frac{1}{2})\quad$ as $\rho(x):= {\rm dist}(x,  \partial\Omega) \to 0$.}
\end{equation}
\end{theorem}

The statement on the interior continuity of $u$ is essentially a consequence of recent regularity estimates by Kassmann and Mimica~\cite{KM}. The boundary decay estimate~\eqref{boundary-decay-main} requires extra work and follows by constructing suitable barrier functions and applying the maximum principle for $\loglap$ on subdomains with small measure. We emphasize that there is no restriction on the measure of $\Omega$ itself in Theorem~\ref{pr 2.1-main}.

The paper is organized as follows. In Section~\ref{sec:basic-properties} we establish basic properties of the logarithmic Laplacian, and we prove Theorem~\ref{teo-representation} and Proposition~\ref{dini-continuity-well-defined-theorem}. In Section~\ref{sec:funct-analyt-fram}, we set up the functional analytic framework for Dirichlet problems related to $\loglap$ in weak sense. Moreover, we prove Theorems~\ref{pr eigen 1}, \ref{sec:funct-analyt-fram-approximation-main} and Corollary~\ref{sec:faber-Krahn-main} in this section. In Section~\ref{sec:maximum-principle}, we discuss the maximum principle for the operator $\loglap$. In particular, we prove Theorem~\ref{sec:maxim-princ-bound-characterization} and Corollaries~\ref{sec:main-result-max-principle} and \ref{non-max-principle-simple}. In Section~\ref{sec:regul-bound-decay}, we consider the interior and boundary regularity of weak solutions of (\ref{eq 3.1-poisson-main}), and we prove Theorem~\ref{pr 2.1-main}. Finally, in an appendix, we prove a logarithmic boundary Hardy inequality related to the function spaces defined in Section~\ref{sec:maximum-principle}. 

Throughout the remainder of the paper, we let $B_r(x) \subset \R^N$ denote the open ball of radius $r$ centered at $x \in \R^N$, and we put $B_r:= B_r(0)$ for $r>0$. Moreover, if a fixed domain $\Omega \subset \R^N$ is considered, we let $\rho(x)= {\rm dist}(x,  \partial\Omega)$ for $x \in \R^N$.

\section{Basic properties}
\label{sec:basic-properties}

We begin this section with the derivation of the integral representation (\ref{representation-main}) of $\loglap$ for functions $u \in C^\beta_c(\R^N)$ as stated in Theorem~\ref{teo-representation}.

\begin{proof}[Proof of Theorem~\ref{teo-representation}]
Let $u\in C^\beta_c(\R^N)$. For $0< \s < \min \{\frac{\beta}{2},\frac{1}{2}\}$, the principal value in the definition of $[(-\Delta)^\s u](x)$ reduces to a standard Lebesgue integral. Let
$R>4$ be chosen such that $\supp u \subset B_{\frac{R}{4}}$. For $x \in \R^N$, we then have
$$
[(-\Delta)^\s u](x)= c_{N,\s}\int_{\R^N}\frac{u(x)-u(x+z)}{|z|^{N+2\s}}\,dz
= A_R(\s,x) + D_R(\s)\,u(x)
$$
with
$$
A_R(\s,x) := c_{N,\s} \Bigl( \int_{B_R}\frac{u(x)-u(x+z)}{|z|^{N+2\s}}\,dz - \int_{\R^N \setminus B_R}\frac{u(x+z)}{|z|^{N+2\s}}\,dz \Bigr)
$$
and
$$
D_R(\s):= c_{N,\s}\int_{\R^N \setminus B_R}|z|^{-N-2\s}dz =  \frac{c_{N,\s} |S^{N-1}|R^{-2\s}}{2\s}.
$$
Here we recall that $B_R := B_R(0)$. If $|x| \ge \frac{R}{2}$, we have that  $u(x)=0$ and $|z| \ge \frac{|x|}{2} \ge 1$ whenever $x+z \in \supp u$, and therefore
$$
|A_R(\s,x)|= c_{N,\s}\Bigl|\int_{\R^N}\frac{u(x+z)}{|z|^{N+2s}}\,dz\Bigr| \le
c_{N,\s} \int_{\R^N}\frac{|u(x+z)|}{|z|^{N}}\,dz \le  2^{N} c_{N,\s}\|u\|_{L^1} |x|^{-N}.
$$
Consequently,
\begin{equation}
  \label{eq:p-a-alpha-est}
\|A_R(\s,\cdot)\|_{L^p(\R^N \setminus B_{\frac{R}{2}})} \le  2^{2N-\frac{N}{p}} c_{N,\s}\|u\|_{L^1}\Bigl(
\frac{|S^{N-1}|}{(p-1)N}\Bigr)^{\frac{1}{p}}{R}^{\frac{N}{p}-N}
\end{equation}
for $1<p< \infty$ and
\begin{equation}
  \label{eq:infty-a-alpha-est}
\|A_R(\s,\cdot)\|_{L^\infty(\R^N \setminus B_{\frac{R}{2}})} \le  4^{N} c_{N,\s}\|u\|_{L^1} {R}^{-N}.
\end{equation}
We now write
\begin{equation}
  \label{eq:def-d-N-s}
c_{N,\s} = \s d_{N}(\s) \qquad \text{with}\qquad d_{N}(\s) := \frac{c_{N,\s}}{\s}= \pi^{-\frac{N}{2}} 2^{2\s}
\frac{\Gamma(\frac{N}{2}+\s)}{\Gamma(1-\s)},
\end{equation}
and we note that
$$
d_N(0) = \pi^{-\frac{N}{2}}\Gamma(\frac{N}{2})= c_N \quad \text{and}\quad d_N'(0)= \pi^{-\frac{N}{2}} \Bigl( (2 \log 2 - \gamma)\Gamma(\frac{N}{2})+ \Gamma'(\frac{N}{2})\Bigr)= c_N \rho_N.
$$
From (\ref{eq:p-a-alpha-est}) and (\ref{eq:infty-a-alpha-est}), we then deduce that
\begin{equation}
  \label{eq:l-p-conv-1}
\|A_R(\s,\cdot)\|_{L^p(\R^N \setminus B_{\frac{R}{2}})} \le \s m_p R^{\frac{N}{p}-N}
\end{equation}
for $1 < p \le \infty$ with a constant $m_p>0$ depending on $u$ but not on $R$ and $\s$.

On the other hand, for $x \in B_{\frac{R}{2}}$, we have $|z| \le B_R$ whenever $x+z \in \supp u$ and therefore the second integral in the definition of $A_R(\s,x)$ vanishes. Since $u \in C^\beta_c(\R^N)$, it is thus easy to see that
\begin{align}
\frac{A_R(\s,x)}{\s} &= d_{N}(\s) \int_{B_R}\frac{u(x)-u(x+z)}{|z|^{N+2\s}}\,dz  \nonumber\\
&\to \tilde A_R(x):= c_N
\int_{B_R}\frac{u(x)-u(x+z)}{|z|^{N}}\,dz  \qquad \text{as $\ \s \to 0^+$,}   \label{eq:l-p-conv-2}
\end{align}
and this convergence is uniform in $x \in B_{\frac{R}{2}}$. Next we note that
$$
\lim_{s \to 0^+}D_R(\s)=  d_{N}(\s) \frac{|S^{N-1}|}{2}R^{-2\s} = \frac{c_N|S^{N-1}|}{2}= 1
$$
and
$$\lim_{s \to 0^+} \frac{D_R(\s)-1}{\s}= \frac{|S^{N-1}|}{2} \bigl(d_N'(0) -2 d_N(0) \log R) = \rho_N - c_N |S^{N-1}| \log R =:\kappa_R.
$$
Since $u \in C_c^\beta(\R^N)$, we then conclude that
\begin{equation}
\lim_{s \to 0^+}\Bigl\|\frac{D_R(\s)u-u}{\s}-\kappa_R u \Bigr\|_{L^p(\R^N)} = 0  \label{eq:l-p-conv-3}
\end{equation}
for $1<p \le \infty$. By (\ref{eq:l-p-conv-2}) and the definition of $\kappa_R$ above, we also find that
\begin{align}
\tilde A_R(x) &+\kappa_R u(x)= c_N \Bigl( \int_{B_R}\frac{u(x)-u(x+z)}{|z|^{N}}\,dz - u(x)\int_{B_R \setminus B_1}\frac{1}{|z|^{N}}\,dz\Bigr) + \rho_N u(x)\nonumber\\
&=c_N \int_{B_R}\frac{ u(x)1_{B_1}(z)-
u(x+z)}{|z|^{N }}  dz + \rho_N u(x) \nonumber\\
&= [\loglap u](x)+ F(x) \qquad \text{with}\quad F(x):=c_N\int_{\R^N \setminus B_R}\frac{u(x+z)}{|z|^{N }}dz \label{eq:l-p-conv-4}
\end{align}
for $x \in \R^N$. By the same estimates as for $A_R(s,x)$, we see that $F(x)=0$ for $|x| \le \frac{R}{2}$ and
\begin{equation}
  \label{eq:estimate-F-x}
\|F\|_{L^P(\R^N \setminus B_{\frac{R}{2}})} \le M_p R^{\frac{N}{p}-N}
\end{equation}
for $1 < p \le \infty$ with a constant $M_p>0$ depending on $u$ but not on $R$ and $\s$.
Combining (\ref{eq:l-p-conv-1}), (\ref{eq:l-p-conv-2}), (\ref{eq:l-p-conv-3}), (\ref{eq:l-p-conv-4}) and (\ref{eq:estimate-F-x}), we see that
$$
\limsup_{\s \to 0^+}
\Bigl\|\frac{(-\Delta)^\s u -u}{\s}- \,\loglap u \Bigr\|_{L^p(\R^N)}  \le (m_p+M_p) R^{\frac{N}{p}-N}
\qquad \text{for every $R>0$, $p \in (1,\infty]$}
$$
and therefore
$$
\lim_{\s \to 0^+}
\Bigl\|\frac{(-\Delta)^\s u -u}{\s}- \,\loglap u \Bigr\|_{L^p(\R^N)}=0\qquad \text{for every $p \in (1,\infty]$.}
$$
In particular, this holds for $p=2$, and therefore, using the continuity of the Fourier transform as a map $L^2(\R^N) \to L^2(\R^N)$, we find that
$$
\widehat{\loglap u} = \lim_{\s \to 0^+}\frac{\widehat{(-\Delta)^\s u} -\hat{u}}{\s}= \lim_{\s \to 0^+}\Bigl(\frac{|\cdot|^{2\s}-1}{\s}\Bigr)\hat{u} = 2 \log|\cdot| \,\hat{u} \quad \text{in \,$L^2(\R^N)$}.
$$
From this we infer that
$$
\widehat{\loglap u }(\xi) = 2 \log|\xi| \,\hat{u}(\xi) \qquad \text{for almost every $\xi \in \R^N$,}
$$
as claimed.
\end{proof}

It is convenient to introduce the kernel functions
\begin{equation}
  \label{eq:def-k}
\k: \R^N \setminus \{0\} \to \R, \qquad \k(z)=c_N 1_{B_1}(z)|z|^{-N}
\end{equation}
and
\begin{equation}
  \label{eq:def-j}
\j: \R^N \to \R,\qquad \j(z)= c_N 1_{\R^N \setminus B_1}(z)|z|^{-N}.
\end{equation}
Then the integral representation (\ref{representation-main}) can be rewritten as
\begin{equation}
  \label{eq:representation-main-rewritten}
\loglap u (x) = \int_{\R^N}(u(x)-u(y))\k(x-y)\,dy - [\j * u](x) + \rho_N u(x).
\end{equation}
We now analyze for which functions $u$ and points $x \in \R^N$, the expression $\loglap u(x)$ is well-defined by this formula. We need to recall some definitions. Let $\Omega \subset \R^N$ be a measurable subset and $u: \Omega \to \R$ be a measurable function. The module of continuity of $u$ at a point $x \in \Omega$
is defined by
$$
\omega_{u,x,\Omega}: (0,\infty) \to [0,\infty),\qquad  \omega_{u,x,\Omega}(r)= \sup_{\stackrel{
y \in \Omega}{|y-x|\le r}} |u(y)-u(x)|.
$$
The function $u$ is called {\em Dini continuous} at $x$ if $\int_0^1 \frac{\omega_{u,x,\Omega}(r)}{r}\,dr < +\infty$.
If
\begin{equation*}
\int_0^1 \frac{\omega_{u,\Omega}(r)}{r}\,dr < \infty \qquad \text{for the uniform continuity module $\omega_{u,\Omega}(r):=
\sup \limits_{x \in \Omega}\omega_{u,x,\Omega}(r)$,}
\end{equation*}
then we call $u$ {\em uniformly Dini continuous} in $\Omega$. In the following, for $s \in \R$, we also let $L^1_s(\R^N)$ denote the space of locally integrable functions $u: \R^N \to \R$ such that
$$
\|u\|_{L^1_s}:= \int_0^\infty \frac{|u(x)|}{(1+|x|)^{N+2s}}\,ds < +\infty.
$$
We need the following observation.

\begin{lemma}
\label{continuity-convolution}
Let $u \in L^1_0(\R^N)$ and  $v: \R^N \to \R$ be measurable with $|v(x)| \le C (1+|x|)^{-N}$ for $x \in \R^N$ with some $C>0$.
Then the convolution $v * u: \R^N \to \R$ is well-defined and continuous.
\end{lemma}

\begin{proof}
This is a direct consequence of Lebesgue's theorem if
$u \in C_c(\R^N)$. Moreover, if $K \subset \R^N$ is compact, we see that
\begin{align*}
|[v * u](x)|&\le \int_{\R^N}|v(x-y)||u(y)|\,dy \le C \int_{\R^N} \frac{|u(y)|}{(1+|x-y|)^N}\,dy \\
&\le C_K \int_{\R^N} \frac{|u(y)|}{(1+|y|)^N}\,dy = C_K \|u\|_{L^1_0} \qquad \text{for $x \in K$ with a constant $C_K>0$},
\end{align*}
and therefore $\|v * u\|_{L^\infty(K)} \le C_K \|u\|_{L^1_0}$. Hence, using the fact that $C_c(\R^N)$ is dense in $L^1_0(\R^N)$, a standard approximation argument shows that $v* u$ is continuous on $K$ for arbitrary $u \in L^1_0(\R^N)$. Since $K$ is chosen arbitrarily, we conclude that $v* u: \R^N \to \R$ is continuous.
\end{proof}

The following is an extension of Proposition~\ref{dini-continuity-well-defined-theorem}.

\begin{proposition}
\label{dini-continuity-well-defined}
Let $u \in L^1_0(\R^N)$.
\begin{enumerate}
\item[i)] If $u$ is Dini continuous at some $x \in \R^N$, then $[\loglap u](x)$ is well-defined by the formula~(\ref{representation-main}). Moreover, if $\Omega \subset \R^N$ is an open subset and $x \in \Omega$, then we have the alternative representation
\begin{equation}
 \loglap  u(x) = c_{N} \int_{\Omega} \frac{ u(x) -u(y)}{|x-y|^{N} } dy - c_N \int_{\R^N \setminus \Omega} \frac{u(y)}{|x-y|^{N}}\,dy + [h_\Omega(x)+\rho_N] u(x), \label{representation-regional}
\end{equation}
where $h_\Omega: \Omega \to \R$ is given by (\ref{eq:def-h-Omega}).
\item[ii)] If $u$ is uniformly Dini continuous in some open subset $\Omega \subset \R^N$, then $\loglap u$ is continuous in $\Omega$.
\end{enumerate}
\end{proposition}

\begin{proof}
i) Since
\begin{align*}
\int_{\R^N  }\frac{|u(x)1_{B_1(x)}(y)-u(y)|}{|x-y|^{N} } dy &= \int_{B_1(x)}\frac{|u(x)-u(y)|}{|x-y|^{N} } dy + \int_{\R^N \setminus B_1(x)}
\frac{|u(y)|}{|x-y|^N}\,dy\\
&\le |S^{N-1}| \int_{0}^1 \frac{\omega_{u,x,\Omega}(r)}{r}\,dr + C_x \int_{\R^N}\frac{|u(y)|}{(1+|y|)^N}\,dy < +\infty
\end{align*}
with a constant $C_x>0$ by assumption, it follows that $ \loglap  u(x)$ is well-defined by (\ref{representation-main}). Next, we let $\Omega \subset \R^N$ be an open subset such that $x \in \Omega$. Starting from (\ref{representation-main}), we see that
\begin{align*}
 &\frac{\loglap  u(x) -\rho_N u(x)}{c_N}=  \int_{\R^N  } \frac{ u(x)1_{B_1(x)}(y)-u(y)}{|x-y|^{N} } dy \\
&= \int_{\Omega } \frac{ u(x)1_{B_1(x)}(y)-u(y)}{|x-y|^{N} } dy +  \int_{\R^N \setminus \Omega } \frac{ u(x)1_{B_1(x)}(y)-u(y)}{|x-y|^{N} } dy \\
&= \int_{\Omega } \frac{ u(x)-u(y)}{|x-y|^{N} } dy   - u(x) \int_{\Omega \setminus B_1(x)} \frac{1}{|x-y|^{N} } dy
+ u(x) \int_{B_1(x) \setminus \Omega} \frac{1}{|x-y|^{N} } dy-  \int_{\R^N \setminus \Omega} \frac{u(y)}{|x-y|^{N}}\,dy\\
&= \int_{\Omega } \frac{u(x)-u(y)}{|x-y|^{N} } dy -  \int_{\R^N \setminus \Omega} \frac{u(y)}{|x-y|^{N}}\,dy + \frac{h_{\Omega}(x) u(x)}{c_N}.
\end{align*}
This yields (\ref{representation-regional}).\\[1mm]
ii) We start with a preliminary remark. Let $\eps_k:= 2^{-2^k}$ for $k \in \N_0$. Then we have 
$$\sum_{k=1}^\infty \omega_{u,\Omega}(\eps_k) \log \frac{1}{\eps_k}
= 2 \sum_{k=1}^\infty \omega_{u,\Omega}(\eps_k) \log \frac{\eps_{k-1}}{\eps_k}
\le 2 \sum_{k=1}^\infty \int_{\eps_k}^{\eps_{k-1}}\frac{\omega_{u,\Omega}(r)}{r}\,dr
= 2 \int_{0}^{\frac{1}{2}} \frac{\omega_{u,\Omega}(r)}{r}\,dr <\infty
$$
by assumption. Hence
\begin{equation}
\label{prelim-remark}
\omega_{u,\Omega}(\eps_k) \log \frac{1}{\eps_k} \to 0 \qquad \text{as $k \to \infty.$}
\end{equation}
Next, we write $\loglap u = c_N f_1 + f_2$ with
$$
f_1,f_2: \R^N \to \R,\qquad   f_1(x)= \int_{B_1(x)} \frac{u(x)-u(y)}{|x-y|^{N}} dy,\qquad f_2(x) = -[\j * u](x) + \rho_N u(x).
$$
By Lemma~\ref{continuity-convolution}, we see that $f_2$ is continuous on $\Omega$. To see the continuity of $f_1$ in $\Omega$, we let $x \in \Omega$,
 $r:= \min \{1, \frac{\dist(x,\partial \Omega)}{4}\}$ and $k \in \N$ be chosen such that $\eps_k <r$. For $y \in \Omega$ with $|x-y|<\eps_k$, we then have
\begin{align*}
&|f_1(x)-f_1(y)| = \Bigl|\int_{B_1} \frac{u(x)-u(x+z) -[u(y)-u(y+z)]}{|z|^{N} } dz \Bigr|\\
&\le \int_{B_{\eps_k}}\frac{|u(x)-u(x+z)|+|u(y)-u(y+z)|}{|z|^{N}}dz +\Bigl| \int_{B_1 \setminus B_{\eps_k}}\frac{u(x)-u(y)+u(x+z)-u(y+z)}{|z|^{N}}dz\Bigr|\\
&\le 2 \int_{B_{\eps_k}} \frac{\omega_{u,\Omega}(|z|)}{|z|^{N}}dz + \omega_{u,\Omega}(|x-y|) \int_{B_1 \setminus B_{\eps_k}}\frac{1}{|z|^{N}}dz
+ \Bigl| \int_{B_1 \setminus B_{\eps_k}} \frac{u(x+z)-u(y+z)}{|z|^{N}}dz\Bigr|
\\
&\le |S^{N-1}|\delta_k +\Bigl|\int_{B_1 \setminus B_{\eps_k}}\frac{u(x+z)-u(y+z)}{|z|^{N}}dz\Bigr|
\end{align*}
with $\delta_k := 2\int_{0}^{\eps_k} \frac{\omega_{u,\Omega}(\tau)}{\tau}\,d\tau  + \omega_{u,\Omega}(\eps_k) \log \frac{1}{\eps_k}$.
We note that $\delta_k \to 0$ by assumption and (\ref{prelim-remark}). Moreover, defining
$$
v_k: \R^N \to \R,\qquad v_k(z)= 1_{B_1 \setminus B_{\eps_k}}(z) |z|^{-N},
$$
we see by Lemma~\ref{continuity-convolution} that
$$
\Bigl|\int_{B_1 \setminus B_{\eps_k}}\frac{u(x+z)-u(y+z)}{|z|^{N}}dz\Bigr|= \bigl|[v_k* u](x)-[v_k* u](y)\bigr|
\to 0 \qquad \text{as $y \to x$}
$$
for every $k \in \N$. We thus conclude that
$$
\limsup_{y \to x}\bigl|f_1(x)-f_1(y)\bigr| \le |S^{N-1}|\delta_k \qquad \text{for every $k \in \N$,}
$$
and this implies that $\lim \limits_{y \to x}|f_1(x)-f_1(y)|= 0$. Hence $f_1$ is continuous in $x$, and so is $\loglap u$.
\end{proof}

\section{Functional analytic framework for the Dirichlet problem}
\label{sec:funct-analyt-fram}

In this section, we set up the functional analytic framework for Dirichlet problems related to $\loglap$ in weak sense.
Throughout this section, let $\Omega$ be a bounded domain. Using (\ref{eq:representation-main-rewritten}), we observe that
\begin{equation}
\int_{\R^N}[\loglap u]v dx =\frac{1}2 \int_{\R^N} \int_{\R^N}(u(x)-u(y))(v(x)-v(y)) \k(x-y)dx dy -  \int_{\R^N}
\bigl(\j * u - \rho_N u\bigr)v\,dx
\label{eq 1.2.2}
\end{equation}
for uniformly Dini continuous functions $u,v \in C_c(\R^N)$ with the kernel functions $\k,\j$ defined in (\ref{eq:def-k}) and (\ref{eq:def-j}). We let $\cH(\Omega)$ denote the space of all measurable functions $u:\R^N\to \R$ with $u \equiv 0$ on $\R^N \setminus \Omega$ and
$$
\int_{\R^N} \int_{\R^N} (u(x)-u(y))^2 \k(x-y) dx dy <+\infty.
$$
Then $\cH(\Omega)$ contains the space of all uniformly Dini continuous functions $u: \R^N \to \R$
with $u \equiv 0$ on $\R^N \setminus \Omega$.
On $\cH(\Omega)$, we define the inner product
$$
\mathcal{E}(u,w)=\frac{1}2 \int_{\R^N}\int_{\R^N} (u(x)-u(y))(w(x)-w(y))\k(x-y)dx dy
$$
and the induced norm $\norm{u}_{\cH(\Omega)}=\sqrt{\mathcal{E}(u,u)}$.
By \cite[Lemma 2.7]{FKV13}, we have that
\begin{equation}
  \label{eq:FKV-eq}
\inf_{u\in \cH(\Omega)}\frac{\mathcal{E}(u,u)}{\|u\|_{L^2(\Omega)}^2}>0,
\end{equation}
and from this it can be deduced that $\cH(\Omega) \subset L^2(\Omega)$ is a Hilbert space. Moreover,
\begin{equation}
  \label{eq:embedding-compact}
\text{the embedding $\cH(\Omega) \hookrightarrow L^2(\Omega)$ is compact}
\end{equation}
by \cite[Theorem 2.1]{CP}. We note that, for $u \in \cH(\Omega)$, we have
\begin{equation}
\cE(u,u)= \frac{1} 2 \int_{\Omega} \int_{\Omega}(u(x)-u(y))^2\k(x-y) dx dy +  \int_\Omega \kappa_{\Omega}(x) u^2(x) \,dx
  \label{eq:alternative-representation-0}
\end{equation}
with
\begin{equation}
  \label{eq:def-kappa-omega}
\kappa_{\Omega}(x) = \int_{\R^N \setminus \Omega}\k(x-y)\,dy = c_N \int_{B_1(x)\setminus \Omega} \frac{1}{|x-y|^N}dy.
\end{equation}
The function $\kappa_\Omega: \Omega \to \R$ is usually called the {\em Killing measure} associated with the kernel $\k$.
For $x \in \Omega$ with $r:= \rho(x)= \dist(x,\partial \Omega) < 1$, we have
\begin{equation}
  \label{eq:upper-estimate-kappa-omega}
\kappa_\Omega(x) \le c_N \int_{B_1(x) \setminus B_r(x)}|x-y|^{-N}\,dy =  c_N |S^{N-1}| \log \frac{1}{r} =
2 \log \frac{1}{\rho(x)}.
\end{equation}
Moreover, if $\partial \Omega$ is Lipschitz, then a standard computation also gives the lower bound
\begin{equation}
  \label{eq:lower-estimate-kappa-omega}
\kappa_\Omega(x) \ge C_\Omega \log \frac{1}{\rho(x)} \qquad \text{for $x \in \Omega$ with a constant $C_\Omega>0$.}
\end{equation}
A more precise estimate is available if $\Omega$ has a uniform $C^2$-boundary, see Remark~\ref{remark-est-killing-measure} below.
The following density result will be useful in the sequel.

\begin{theorem}\label{density}
Let $\Omega \subset \R^N$ be a bounded Lipschitz domain. Then $C^\infty_c(\Omega)$ is dense in $\cH(\Omega)$. Moreover, if $u \in \cH(\Omega)$ is nonnegative, we have:
\begin{itemize}
\item[(i)] There exists a sequence of functions $u_n \in \cH(\Omega)$ with compact support in $\Omega$, $0 \le u_n \le u_{n+1} \le u$ for all $n \in \N$ and
 $u_n \to u$ in $\cH(\Omega)$ as $n\to+\infty$.
\item[(ii)] There exists a sequence of nonnegative functions $u_n \in C^\infty_c(\Omega)$, $n \in \N$ with $u_n \to u$ in $\cH(\Omega)$ as $n\to+\infty$.
\end{itemize}
\end{theorem}

\begin{proof}
Since for $u \in \cH(\Omega)$ we have $u^\pm \in \cH(\Omega)$, it suffices to consider a nonnegative function $u \in \cH(\Omega)$ and to prove (i) and (ii).\\
(i) For $r>0$, we define the Lipschitz function
$$
\phi_r: \R^N \to \R, \qquad \phi_r(x) = \left \{
  \begin{aligned}
   &0,&&\quad \quad  \rho(x) \ge 2r,\\
   &2 - \frac{\rho(x) }{r} ,&&\quad \quad  r \le \rho(x) \le 2r,\\
   &1,&& \quad \quad \rho(x) \le r.
  \end{aligned}
\right.
$$
Clearly, we then have that $\phi_s \le \phi_r$ whenever $0 < s \le r$. We shall show that
\begin{equation}
  \label{eq:u-phi-r-claim}
u \phi_r \in \cH(\Omega)\; \text{for $r>0$ sufficiently small}\qquad \text{and}\qquad \|u \phi_r\|_{\cH(\Omega)} \to 0\quad \text{as $r \to 0$.}
\end{equation}
As a consequence, we then have $u(1-\phi_r) \in \cH(\Omega)$ for $r>0$ sufficiently small and $u(1-\phi_r) \to u$ in $\cH(\Omega)$ as $r \to 0$. We also note that $0 \le u(1-\phi_r) \le u(1-\phi_s) \le u$ for $0<s \le r$, and that $u(1-\phi_r)$ has compact support in $\Omega$. Hence the claim follows by choosing $u_n:= u (1-\phi_{r_n})$ for a decreasing sequence of numbers $r_n>0$ with $\lim \limits_{n \to \infty}r_n = 0$.

To prove (\ref{eq:u-phi-r-claim}), we consider constants $C>0$ which may change from line to line. Let $A_t:= \{x \in \Omega\:: \rho(x) \le t\}$ for $t >0$. Since $u \phi_r$ vanishes in $\R^N \setminus A_{2r}$, $0 \le \phi \le 1$ on $\R^N$ and
$$
|{\phi_r}(x)-{\phi_r}(y)| \le \min \{\frac{|x-y|}{r},1\} \qquad \text{for $x,y \in \R^N$,}
$$
we have that for $r\in(0,\frac14)$
	\begin{align*}
	&\|{\phi_r} u\|^2_{\cH(\Omega)} =\cE(\phi_r u, \phi_r u)= \frac{1}{2}\int_{\R^N}\int_{\R^N} [u(x){\phi_r}(x)-u(y){\phi_r}(y)]^2\k(x-y)\ dydx\\
        &= \frac{1}{2}\int_{A_{4r}}\int_{A_{4r}} [u(x){\phi_r}(x)-u(y){\phi_r}(y)]^2\k(x-y)\ dydx +\int_{A_{2r}} u(x)^2 {\phi_r}(x)^2 \int_{\R^N \setminus A_{4r}} \k(x-y)\ dydx\\
&\le \frac{1}{2}\int_{A_{4r}}\int_{A_{4r}} \bigl[u(x)\bigl({\phi_r}(x)-{\phi_r}(y)\bigr)+ {\phi_r}(y)\bigl(u(x)-u(y)\bigr)\bigr]^2\k(x-y)\ dydx\\
&\quad+ C  \int_{A_{2r}} u(x)^2  \log \frac{1}{\dist(x,\R^N \setminus A_{4r})}dx\\
&\le \int_{A_{4r}} u^2(x) \int_{A_{4r}}({\phi_r}(x)-{\phi_r}(y))^2 \k(x-y)\ dydx
+ \int_{A_{4r}}\int_{A_{4r}}(u(x)-u(y))^2 \k(x-y)\ dydx \\
&\quad+ C\int_{A_{2r}} u(x)^2 \log \frac{1}{\rho(x)}dx
\\&\le \frac{C}{r^2} \int_{A_{4r}} u^2(x)\int_{B_r(x)} |x-y|^{2-N} dydx + C \int_{A_{4r}} u^2(x)\int_{B_1(x) \setminus B_r(x)} |x-y|^{-N} dydx\\
&\quad+ \frac{1}{2}\int_{A_{4r}}\int_{A_{4r}}(u(x)-u(y))^2 \k(x-y)\ dydx + C\int_{A_{2r}} u(x)^2 \log \frac{1}{\rho(x)}dx
	\end{align*}
	\begin{align*}
&\le C \int_{A_{4r}}\Bigl(1+ \log \frac{1}{r} \Bigr) u^2(x) dx
+  \frac{1}{2}\int_{A_{4r}}\int_{A_{4r}}(u(x)-u(y))^2 \k(x-y)\ dydx \qquad\qquad\qquad\qquad\\
&\quad+ C\int_{A_{2r}} u(x)^2 \log \frac{1}{\rho(x)}dx\\
&\le C \int_{A_{4r}}u^2(x) \log \frac{1}{\rho(x)}dx
+ \frac{1}{2}\int_{A_{4r}}\int_{A_{4r}}(u(x)-u(y))^2 \k(x-y)\ dydx.
	\end{align*}
Next we note that
$$
\int_{\Omega }u^2(x) \kappa_\Omega(x)dx <+\infty
$$
by (\ref{eq:alternative-representation-0}), and therefore
$$
\int_{A_{4r}}u^2(x) \log \frac{1}{\rho(x)}dx  \to 0 \qquad \text{as $r \to 0$}
$$
by (\ref{eq:lower-estimate-kappa-omega}). Here we used the fact that $\Omega$ has a Lipschitz boundary.
Moreover, since $\cE(u,u)< +\infty$, we have
$$
\int_{A_{4r}}\int_{A_{4r}}(u(x)-u(y))^2 \k(x-y)\ dydx \to 0 \qquad \text{as $r \to 0$.}
$$
We thus obtain (\ref{eq:u-phi-r-claim}), and this finishes the proof of (i).\\[0.5mm]
(ii) By (i), we may assume that $u \in \cH(\Omega)$ is nonnegative function with compact support in $\Omega$.
Let $u_\eps: \R^N \to \R$ denote the usual mollification of $u$ given by
$$
u_{\epsilon}:=\eta_{\epsilon}\ast u=\int_{\R^N}\eta_{\epsilon}(z)u(\cdot-z)\ dz,
$$
where $\eta_0 \in C^\infty_c(\R^N)$ is a nonnegative radial function supported in $B_1$ with $\int_{\R^N}\eta_0(x)\,dx = 1$ and
$\eta_{\epsilon}(x)=\epsilon^{-N}\eta(\frac{x}{\eps})$ for $x \in \R^N$, $\eps>0$. By \cite[Proposition 4.1]{JW-preprint}, we then have that $u_\eps \in C_c^\infty(\Omega)$ for $\eps>0$ sufficiently small and
$$
\lim_{\eps \to 0} \|u-u_\eps\|_{\cH(\Omega)}^2 = \lim_{\eps \to 0} \cE(u-u_\eps,u-u_\eps) =0.
$$
Hence it suffices to choose $u_n:= u_{\eps_n}$ for a sequence $\eps_n \to 0^+$. The proof is thus finished.
\end{proof}

The Hilbert space $\cH(\Omega)$ provides the appropriate framework to study the Poisson problem
\begin{equation}\label{eq 3.1-poisson}
\left\{ \arraycolsep=1pt
\begin{array}{lll}
 \loglap u=f\quad \  {\rm in}\quad   \Omega,\\[2mm]
 \phantom{ \loglap   }
  u=0\quad \ {\rm{in}}\  \quad \R^N\setminus \Omega
\end{array}
\right.
\end{equation}
with given $f \in L^2(\Omega)$ in weak sense. For this it is convenient to introduce the quadratic form
$$
\cE_L: \cH(\Omega) \times \cH(\Omega) \to \R, \qquad \cE_L(u,w)= \mathcal{E}(u,w) - \int_{\R^N} \bigl( \j * u  -\rho_N u \bigr)w\,dx.
$$
Here we note that, since $\j \in L^r(\R^N)$ for $r>1$ and $\Omega$ is bounded, 
\begin{equation}
  \label{eq:continuous-L-2-weak-L-1}
\text{the map $L^2(\R^N) \to L^2(\Omega)$, $u \mapsto \j * u$ is well-defined and continuous,}
\end{equation}
and thus $\cE_L$ is well-defined. Theorem~\ref{density} in particular implies that, if $\Omega \subset \R^N$ is a bounded Lipschitz domain, then the space of uniformly Dini continuous functions $u \in C_c(\Omega)$ is dense in $\cH(\Omega)$. It thus follows from (\ref{eq 1.2.2}) that
\begin{equation}
  \label{eq:simple-integration-by-parts}
\int_{\R^N}  [\loglap  u] v dx = \mathcal{E}_L(u,w)
\end{equation}
for every uniformly Dini continuous function $u \in C_c(\R^N)$ and every $v \in \cH(\Omega)$. Now, by definition, a function $u \in \cH(\Omega)$ is a weak solution of (\ref{eq 3.1-poisson}) if
$$
\cE_L(u,w) = \int_{\Omega} f w\,dx  \qquad \text{for all $w \in \cH(\Omega)$.}
$$
Since $\cE_L$ contains competing nonlocal terms of different sign, at first glance it appears difficult to compare the values of $\cE_L(u,u)$ and $\cE_L(|u|,|u|)$ for $u \in \cH(\Omega)$. For this, an alternative representation of $\cE_L$ will turn out to be useful.

\begin{proposition}
\label{E-L-alternative-representation}
For $u \in \cH(\Omega)$, we have
\begin{equation}
  \label{eq:alternative-representation}
\cE_L(u,u) = \frac{c_N} 2 \int_{\Omega} \int_{\Omega}\frac{ [u(x)-u(y)]^2}{|x-y|^N}dx dy +  \int_\Omega [h_{\Omega}(x)+\rho_N] u^2(x) \,dx
\end{equation}
with $h_\Omega$ given in (\ref{eq:def-h-Omega}). Moreover, if $\partial \Omega$ is Lipschitz, we have
\begin{equation}
  \label{eq:lower-estimate-h-omega}
h_\Omega(x) \ge C_\Omega \log \frac{1}{\rho(x)}-\tilde C_\Omega \qquad \text{for $x \in \Omega$ with constants $C_\Omega,\, \tilde C_\Omega >0$.}
\end{equation}
\end{proposition}

\begin{proof}
We have
  \begin{align*}
 &\cE_L(u,u)-\rho_N \int_{\Omega}u^2(x)\,dx  = \cE(u,u) -\int_{\R^N} [\j* u] u \,dx\\
&= \frac{c_N}2 \int\int_{\stackrel{x,y \in \Omega}{|x-y|<1}}\frac{ [u(x)-u(y)]^2}{|x-y|^N}dx dy
+ c_N \int_\Omega u^2(x) \int_{B_1(x)\setminus \Omega} \frac{1}{|x-y|^N}dydx\\
&\quad-c_N\int \!\!\!\! \int_{|x-y|\ge1}\frac{u(x)u(y)}{|x-y|^N}dy dx\\
&= \frac{c_N}2 \int\int_{\stackrel{x,y \in \Omega}{|x-y|<1}}\frac{ [u(x)-u(y)]^2}{|x-y|^N}dx dy
+ c_N \int_\Omega u^2(x) \int_{B_1(x)\setminus \Omega} \frac{1}{|x-y|^N}dydx\\
&\quad+\frac{c_N}{2}\int \!\!\!\! \int_{\stackrel{x,y \in \Omega}{|x-y|\ge1}}\frac{(u(x)-u(y))^2}{|x-y|^N}dy dx -
c_N \int \!\!\!\! \int_{\stackrel{x,y \in \Omega}{|x-y|\ge1}}\frac{u^2(x)}{|x-y|^N}dy dx\\
&= \frac{c_N}2 \int_{\Omega} \int_{\Omega}\frac{ [u(x)-u(y)]^2}{|x-y|^N}dx dy + c_N \int_\Omega u^2(x) \Bigl[\int_{B_1(x)\setminus \Omega} \frac{1}{|x-y|^N}dy - \int_{\Omega \setminus B_1(x)} \frac{1}{|x-y|^N}dy\Bigr]dx\\
&= \frac{c_N} 2 \int_{\Omega} \int_{\Omega}\frac{ [u(x)-u(y)]^2}{|x-y|^N}dx dy + c_N \int_\Omega u^2(x) h_{\Omega}(x)\,dx.
\end{align*}
Moreover, (\ref{eq:lower-estimate-h-omega}) is a direct consequence of (\ref{eq:lower-estimate-kappa-omega}).
\end{proof}

The following is a rather immediate corollary of Proposition~\ref{E-L-alternative-representation}.

\begin{lemma}
\label{sec:funct-analyt-fram-1-modulus}
 For $u \in \cH(\Omega)$, we also have $|u| \in \cH(\Omega)$, and
\begin{equation}
  \label{eq:modulus-invariance}
\cE_L(|u|,|u|)\le \cE_L(u,u).
\end{equation}
Moreover, equality holds in (\ref{eq:modulus-invariance}) if and only if $u$ does not change sign.
\end{lemma}

Our next aim is to study the eigenvalue problem
\begin{equation}\label{eq 3.1}
\left\{ \arraycolsep=1pt
\begin{array}{lll}
 \loglap u=\lambda  u\quad \  {\rm in}\quad   \Omega,\\[2mm]
 \phantom{ \loglap   }
  u=0\quad \ {\rm{in}}\  \quad \R^N\setminus \Omega.
\end{array}
\right.
\end{equation}
We consider corresponding eigenfunctions in weak sense, i.e., as weak solutions of (\ref{eq 3.1-poisson}) with $f = \lambda u$. We restate Theorem~\ref{pr eigen 1} from the introduction for the reader's convenience.

\begin{theorem}\label{pr eigen 1-section}
Let $\Omega$ be a bounded domain  in $\R^N$.  Then problem (\ref{eq 3.1-main-results}) admits a sequence of eigenvalues
$$
\lambda_1^L(\Omega)<\lambda_2^L(\Omega)\le \cdots\le \lambda_k^L(\Omega)\le \lambda_{k+1}^L(\Omega)\le \cdots
$$
and corresponding eigenfunctions $\xi_k$, $k \in \N$ such that the following holds:
\begin{enumerate}
\item[(i)] $\lambda_{k}^L(\Omega)=\min \{\cE_L(u,u) \::\:  u\in \cH_k(\Omega)\::\: \norm{u}_{L^2(\Omega)}=1\}$, where
$$
\cH_1(\Omega):= \cH(\Omega)\quad \text{and}\quad  \cH_k(\Omega):=\{u\in\cH(\Omega)\::\: \text{$\int_{\Omega} u \xi_i \,dx =0$ for $i=1,\dots k-1$}\}\quad \text{for $k>1$.}
$$
\item[(ii)] $\{\xi_k\::\: k \in \N\}$ is an orthonormal basis of $L^2(\Omega)$.
\item[(iii)] $\xi_1$ is strictly positive in $\Omega$. Moreover, $\lambda_1^L(\Omega)$ is simple, i.e., if $u \in \cH(\Omega)$ satisfies (\ref{eq 3.1}) in weak sense with $\lambda = \lambda_1^L(\Omega)$, then $u=t\xi_1$ for some $t\in\R$.
\item[(iv)] $\lim \limits_{k \to \infty} \lambda_k^L(\Omega)=+\infty$.
\end{enumerate}
\end{theorem}

\begin{proof}
By (\ref{eq:continuous-L-2-weak-L-1}),
 \begin{equation}
  \label{eq:continuity-conv-functional}
\text{the functional
 $L^2(\Omega) \to \R$, $u \mapsto \int_{\R^N} [\j * u] u\,dx$ is continuous.}
\end{equation}
Consequently, by (\ref{eq:embedding-compact}), the functional
$$
\cH(\Omega) \to \R, \qquad u \mapsto \Phi(u):= \cE_L(u,u)
$$
is weakly lower semicontinuous. Moreover, setting $\cM_1:= \{u\in \cH(\Omega),\, \norm{u}_{L^2(\Omega)}=1\}$, we have that
\begin{equation}
  \label{eq:def-m-1}
m_1:= \sup \limits_{u \in \cM_1}\int_{\R^N} [\j * u]u\,dx < +\infty
\end{equation}
by (\ref{eq:continuity-conv-functional}), which implies that
\begin{equation}
  \label{eq:relative-functional-est}
\cE(u,u) \le \Phi(u) + m_1- \rho_N < +\infty \qquad \text{for $u \in \cM_1$.}
\end{equation}
Put $\lambda_1^L(\Omega):= \inf_{\cM_1} \Phi$.  From (\ref{eq:continuity-conv-functional}), (\ref{eq:relative-functional-est}) and the weak lower semicontinuity of $\Phi$, it then follows that $\lambda_1^L(\Omega)$ is attained by a function $\xi_1 \in \cM_1$. Consequently, there exists a Lagrange multiplier $\lambda \in \R$ such that
$$
\cE_L(\xi_1, \varphi)=\frac{1}{2}\Phi'(u)\varphi = \lambda \int_\Omega \xi_1 \varphi\, dx \qquad \text{for all $\varphi\in \cH(\Omega)$.}
$$
Choosing $\varphi = \xi_1$ yields $\lambda= \lambda_1^L(\Omega)$, hence $\xi_1$ is an eigenfunction of (\ref{eq 3.1}) corresponding to the eigenvalue $\lambda_1^L(\Omega)$. Next we proceed inductively and assume that $\xi_2,\dots,\xi_k \in \cH(\Omega)$ and $\lambda_2^L(\Omega) \le \dots \le \lambda_k^L(\Omega)$ are already given for some $k \in \N$ with the properties that for $i=2,\dots,k$, the function $\xi_i$ is a minimizer
of $\Phi$ within the set
$$
\cM_i:= \{u\in \cH_i(\Omega)\::\: \norm{u}_{L^2(\Omega)}=1\} = \{u\in \cH(\Omega)\::\: \norm{u}_{L^2(\Omega)}=1,\: \text{$\int_{\Omega} u \xi_j \,dx =0$ for $j=1,\dots i-1$} \},
$$
$\lambda_i^L(\Omega)= \inf_{\cM_i} \Phi = \Phi(\xi_i)$, and
\begin{equation}
  \label{eq:inductive-eigenvalue}
  \cE_L(\xi_i, \varphi)=\lambda_i^L(\Omega)\int_\Omega \xi_i \varphi\, dx \qquad \text{for all $\varphi\in \mathbb{H}(\Omega)$.}
\end{equation}
We then put
\begin{align*}
&\cH_{k+1}(\Omega):= \left\{u\in \cH(\Omega)\::\: \norm{u}_{L^2(\Omega)}=1,\: \text{$\int_{\Omega} u \xi_i \,dx =0$ for $i=1,\dots k$} \right\},\\
&\cM_{k+1}:= \left\{u\in \cH_{k+1}(\Omega),\, \norm{u}_{L^2(\Omega)}=1\right\}\quad \text{and}\\
&\lambda_{k+1}^L(\Omega):= \inf_{\cM_{k+1}}\Phi.
\end{align*}
By the same weak lower semicontinuity argument as above, the value $\lambda_{k+1}^L(\Omega)$ is attained by a function $\xi_{k+1} \in \cM_{k+1}$. Consequently, there exists a Lagrange multiplier $\lambda \in \R$ with the property that
\begin{equation}
  \label{eq:inductive-eigenvalue-k+1}
\cE_L(\xi_{k+1}, \varphi)=\lambda \int_\Omega \xi_{k+1} \varphi\, dx \qquad \text{for all $\varphi \in \cH_{k+1}(\Omega)$.}
\end{equation}
Choosing $\varphi = \xi_{k+1}$, it yields $\lambda= \lambda_{k+1}^L(\Omega)$. Moreover, for $i=1,\dots,k$, we have, by (\ref{eq:inductive-eigenvalue}) and the definition of $\cH_{k+1}(\Omega)$,
$$
\cE_L(\xi_{k+1}, \xi_i)=\cE_L(\xi_i, \xi_{k+1})= \lambda_i^L(\Omega)  \int_{\Omega} \xi_i \xi_{k+1}\,dx = 0 = \lambda_{k+1}^L(\Omega)\int_{\Omega} \xi_{k+1} \xi_i \,dx.
$$
Hence (\ref{eq:inductive-eigenvalue-k+1}) holds with $\lambda= \lambda_{k+1}^L(\Omega)$ for all $\varphi\in \cH(\Omega)$.
Inductively, we have now constructed an $L^2$-normalized sequence $(\xi_k)_k$ in $\cH(\Omega)$ and a nondecreasing sequence $(\lambda_k^L(\Omega))_k$ in $\R$ such that property (i) holds and such that $\xi_k$ is an eigenfunction of (\ref{eq 3.1}) corresponding to $\lambda = \lambda_k^L(\Omega)$ for every $k \in \N$. Moreover, by construction, the sequence $(\xi_k)_k$ forms an orthonormal system in $L^2(\Omega)$.
Next we show property (iv), i.e., $\lim \limits_{k\to+\infty} \lambda_k^L(\Omega)=+\infty.$ Supposing by contradiction that $c:= \lim \limits_{k \to \infty}\lambda_k^L(\Omega) <+\infty$, we deduce that
$\cE_L(\xi_k,\xi_k) \le c$ for every $k \in \N$ and therefore
$$
\cE(\xi_k,\xi_k) \le c + m_1- \rho_N < +\infty \qquad \text{for every $k \in \N$}
$$
by (\ref{eq:relative-functional-est}). Hence the sequence $(\xi_k)$ is bounded in $\cH(\Omega)$, and therefore it
contains a convergent subsequence $(\xi_{k_j})_j$ in $L^2(\Omega)$ by (\ref{eq:embedding-compact}).
This however is impossible since the functions $\{\xi_{k_j}\}_{j \in \N}$ are $L^2$-orthonormal. Hence (iv) is proved.

Next, to prove that $\{\xi_k\::\: k \in \N\}$ is an orthonormal basis of $L^2(\Omega)$, we first suppose by contradiction that there exists $v \in \mathbb{H}(\Omega)$ with $\|v\|_{L^2(\Omega)} = 1$
 and $\int_{\Omega} v \xi_k dx =0$ for any $k \in \N$.
Since $\lim \limits_{k\to\infty} \lambda_k^L(\Omega)=+\infty$, there exists an integer $k_0>0$  such that
$\Phi(v)<\lambda_{k_0}^L(\Omega)=\inf_{\cM_{k_0}}\Phi(u)$, which by definition of $\cM_{k_0}$ implies that
$\int_{\Omega} v \xi_k dx \not = 0$ for some $k \in \{1,\dots,k_0-1\}$. Contradiction. We conclude that $\cH(\Omega)$ is contained in the $L^2$-closure of the span of $\{\xi_k\::\: k \in \N\}$. Since $\cH(\Omega)$ is dense in $L^2(\Omega)$, we conclude that the span of $\{\xi_k\::\: k \in \N\}$ is dense in $L^2(\Omega)$, and hence $\{\xi_k\::\: k \in \N\}$ is an orthonormal basis of $L^2(\Omega)$.

Finally, we show property (iii), which also implies that $\lambda_1^L(\Omega)< \lambda_2^L(\Omega)$. Let $w \in \cH(\Omega)$ be a $L^2$-normalized eigenfunction of (\ref{eq 3.1}) corresponding to the eigenvalue $\lambda_1^L(\Omega)$, i.e. we have
\begin{equation}
  \label{eq:inductive-eigenvalue-k=1}
 \cE_L(w, \varphi)=\lambda_1^L(\Omega) \int_\Omega w \varphi\, dx \qquad \text{for all $w \in \mathbb{H}(\Omega)$.}
\end{equation}
We show that $w$ does not change sign. Indeed, choosing $\phi= w$ in (\ref{eq:inductive-eigenvalue-k=1}), we see that $w$ is a minimizer of $\Phi\big|_{\cM_1}$. On the other hand, we also have $|w| \in \cM_1$ and
$$
\Phi(|w|) =\cE_L(|w|,|w|)\le \cE_L(w,w) = \Phi(w) = \lambda_1^L(\Omega),
$$
by Lemma~\ref{sec:funct-analyt-fram-1-modulus}. Hence equality holds by definition of $\lambda_1^L(\Omega)$, and then Lemma~\ref{sec:funct-analyt-fram-1-modulus} implies that $w$ does not change sign. In particular, we may assume that $\xi_1$ is nonnegative, which implies that
$$
\cE(\xi_1,\phi) -[\rho_N+\lambda_1] \int_{\Omega} \xi_1 \phi\,dx = \int_{\Omega} [\j * \xi_1] \phi \,dx \ge 0 \qquad \text{for all $\phi \in C^\infty_c(\Omega)$, $\phi \ge 0$.}
$$
Hence $\xi_1$ is a nontrivial, nonnegative weak supersolution of the equation $I \xi_1 - (\rho_N+\lambda_1)\xi_1=0$ in $\Omega$ in the sense of \cite{JW-preprint}, where $I$ is the integral operator associated with the kernel $\k$ defined in (\ref{eq:def-k}). Therefore, \cite[Theorem 1.1]{JW-preprint} applies and yields that $\xi_1>0$ in $\Omega$.

Now suppose by contradiction that there is a function $w \in \cH(\Omega)$ satisfying (\ref{eq:inductive-eigenvalue-k=1}) and such that $w \not = t\xi_1$ for every $t \in \R$. Then there exist a linear combination $\tilde w$ of $w$ and $\xi_1$ which changes sign, and we may also assume that $\tilde w$ is $L^2$-normalized. Since $\tilde w$ also satisfies (\ref{eq:inductive-eigenvalue-k=1}) in place of $w$, we arrive at a contradiction. Thus the eigenvalue $\lambda_1^L(\Omega)$ is simple, and property (iii) holds.
\end{proof}

Note that the logarithmic Laplacian has the same structure of Dirichlet eigenvalues and eigenfunctions as the fractional Laplacian, see \cite[Section 3]{SV} (cf. also \cite{GS1}). A remarkable relationship between the first Dirichlet eigenvalue of
$\loglap$ and $(-\Delta)^s$ for $s>0$ close to $0$ is given by Theorem~\ref{sec:funct-analyt-fram-approximation-main}, which we prove now and restate here for the reader's convenience.

\begin{theorem}
\label{sec:funct-analyt-fram-approximation}
Let $\Omega$ be a bounded Lipschitz domain  in $\R^N$, and let $\lambda_1^s(\Omega)$ denote the first Dirichlet eigenvalue of $(-\Delta)^s$ on $\Omega$ for $s \in (0,1)$. Then we have
\begin{equation}
  \label{eq:zero-deriv-lambda}
\lambda_1^L(\Omega) = \frac{d}{ds}\Big|_{s=0} \lambda_1^s(\Omega).
\end{equation}
Moreover, if $u_s$ is the unique nonnegative $L^2$-normalized Dirichlet eigenfunction
of $(-\Delta)^s$ corresponding to $\lambda_1^s(\Omega)$, then we have
\begin{equation}
  \label{eq:eigenfunction-convergence}
u_s \to \xi_1 \qquad \text{in $L^2(\Omega)$,}
\end{equation}
where $\xi_1$ is the corresponding unique nonnegative $L^2$-normalized eigenfunction of $\loglap$ corresponding to $\lambda_1^L$.
\end{theorem}

\begin{proof}
We first recall that, for $0<s<1$, the first Dirichlet eigenvalue of $(-\Delta)^s$ is given by
\begin{equation}
\label{variational-s-characterization}
\lambda_1^s(\Omega) = \inf \{\cE_s(u,u)\::\: u \in C^2_c(\Omega),\: \|u\|_{L^2}=1\},
\end{equation}
where the quadratic form $\cE_s$ is defined by
$$
(u,v) \mapsto \cE_s(u,u) = \frac{c_{N,s}}2 \int_{\R^N} \int_{\R^N}\frac{ \bigl(u(x)-u(y)\bigr)\bigl(v(x)-v(y)\bigr)}{|x-y|^{N+2s}}dx dy.
$$
From the variational characterization (\ref{variational-s-characterization}) and the fact that $(-\Delta)^s \psi \to \psi$ as $s \to 0$ for $\psi \in C^2_c(\R^N)$,
we observe that
$$
\lim_{s \to 0^+}  \lambda_1^s(\Omega) = 1.
$$
Let
$$
\Gamma_\s: C_c^2(\R^N) \to L^2(\R^N),\qquad \Gamma_\s u = \frac{(-\Delta)^\s u - u}{\s}.
$$
For $w \in C^2_c(\Omega)$ with $\|w\|_{L^2}=1$, we have that
$$
\limsup_{s \to 0^+} \frac{ \lambda_1^s(\Omega)-1}{s} \le \limsup_{s \to 0^+}\frac{
\cE_s(w,w)-\|w\|^2_{L^2}}{s}= \lim_{s \to 0^+} \langle \Gamma_s w, w\rangle_{L^2(\Omega)}
= \langle \loglap w, w\rangle_{L^2(\Omega)}
$$
by Theorem~\ref{teo-representation}(i), and consequently,
\begin{equation}
  \label{eq:limsup-est}
\limsup_{s \to 0^+} \frac{ \lambda_1^s(\Omega)-1}{s} \le \inf_{\stackrel{w \in C^2_c(\Omega)}{\|w\|_{L^2}=1}} \langle \loglap w, w\rangle_{L^2(\Omega)}
= \inf_{\stackrel{w \in \cH(\Omega)}{\|w\|_{L^2}=1}} \cE_L(w,w)= \lambda_1^L(\Omega).
\end{equation}
Here we used~(\ref{eq:simple-integration-by-parts}) and the fact that $C^2_c(\Omega)$ is dense in $\cH(\Omega)$ by Theorem~\ref{density}. Next we wish to prove (\ref{eq:eigenfunction-convergence}). For this we first prove that the functions $u_{s}$ remain uniformly bounded in $\cH(\Omega)$ as $s \to 0^+$. Indeed, by  (\ref{eq:limsup-est}) we have, as $s \to 0^+$,
\begin{align*}
&\lambda_1^L(\Omega) +o(1) \ge \frac{\lambda_1^{s}(\Omega) -1}{s} = \frac{\cE_s(u_s,u_s)-1}{s} \\
& =\frac{c_{N,s}}{2s} \int \!\!\!\! \int_{|x-y| \le 1} \frac{(u_s(x)-u_s(y))^2}{|x-y|^{N+2s}}dxdy  - \frac{c_{N,s}}{s}
\int \!\!\!\! \int_{|x-y| \ge 1} \frac{u_s(x)u_s(y)}{|x-y|^{N+2s}}dxdy\\
&\quad+\frac{1}{s}\Bigl(c_{N,s}
\int_{\Omega} \int_{\R^N \setminus B_1(x)} \frac{u_s(x)^2}{|x-y|^{N+2s}}dydx -1\Bigr)\\
& =\frac{c_{N,s}}{2s} \int \!\!\!\! \int_{|x-y| \le 1} \frac{(u_s(x)-u_s(y))^2}{|x-y|^{N+2s}}dxdy  - \frac{c_{N,s}}{s}
\int \!\!\!\! \int_{|x-y| \ge 1} \frac{u_s(x)u_s(y)}{|x-y|^{N+2s}}dxdy + f_N(s)
\end{align*}
with
$$
f_N(s)= \frac{1}{s} \Bigl( c_{N,s} \int_{\R^N \setminus B_1(x)}|x-y|^{-N-2s}\,dy -1\Bigr)= \frac{1}{s}\Bigl(\frac{c_{N,s}|S^{N-1}|}{2s}-1\Bigr).
$$
Here we used that $\|u_s\|_{L^2(\Omega)}=1$. Consequently,
\begin{align*}
\lambda_1^L(\Omega) +o(1) &\ge \frac{c_{N,s}}{2s} \int \!\!\!\! \int_{|x-y| \le 1} \frac{(u_s(x)-u_s(y))^2}{|x-y|^{N}}dxdy  - \frac{c_{N,s}}{s}
\int \!\!\!\! \int_{|x-y| \ge 1} \frac{u_s(x)u_s(y)}{|x-y|^{N}}dxdy + f_N(s)\\
&= \frac{c_{N,s}}{ s c_N}\Bigl( \cE(u_s,u_s)  - \int_{\Omega} [\j * u_s] u_s dx\Bigr) + f_N(s).
\end{align*}
Therefore, writing again $c_{N,s}= s d_N(s)$ as in (\ref{eq:def-d-N-s}) and using that
$$
d_N(s) \to c_N, \quad f_{N}(s) =  \frac{1}{s}\Bigl(\frac{d_N(s) |S^{N-1}|}{2}-1\Bigr) \to \frac{|S^{N-1}|}{2} d_N'(0)= \frac{d_N'(0)}{c_N}= \rho_N \;\quad \text{as \,$s \to 0^+$,}
$$
we infer that
\begin{align*}
\cE(u_s,u_s) &\le \frac{c_N}{d_{N}(s)}\Bigl( \lambda_1^L(\Omega) +o(1) - f_N(s) \Bigr)  + \int_{\Omega} [\j * u_s] u_s dx  \\
&\le \Bigl(1+o(1))\Bigl( \lambda_1^L(\Omega) - \rho_N  +o(1) \Bigr) + m_1 \qquad \text{as $s \to 0^+$,}
\end{align*}
where $m_1$ is given in (\ref{eq:def-m-1}). We thus conclude that the functions $u_{s}$ remain uniformly bounded in $\cH(\Omega)$ as $s \to 0^+$.
Now, to prove (\ref{eq:eigenfunction-convergence}), we argue by contradiction and suppose that there exists $\eps>0$ and a sequence $s_k \to 0$ with
\begin{equation}
  \label{eq:contradiction-statement}
\|u_{s_k}-\xi_1\|_{L^2(\Omega)} \ge \eps \qquad \text{for all $k \in \N$.}
\end{equation}
Since the sequence $\{u_{s_k}\}_k$ is bounded in $\cH(\Omega)$, we may, by (\ref{eq:embedding-compact}), pass to a subsequence such that
$$
\text{$u_{s_k} \rightharpoonup u_0$ in $\cH(\Omega)$,}\quad
\text{$u_{s_k} \to u_0$ in $L^2(\Omega)$}\quad \text{and}\quad
\text{$\frac{\lambda_1^{s_k}(\Omega)-1}{s_k} \to \lambda^* \in [-\infty,\lambda_1^L(\Omega)]$}\qquad \text{as $k \to \infty$.}
$$
In particular, it follows that $\|u_0\|_{L^2}=1$. For $\psi \in C^2_c(\Omega)$, we now find that
\begin{align}
\lim_{k \to \infty}\frac{\lambda_1^{s_k}(\Omega)-1}{s_k} \langle u_{s_k}, \psi \rangle_{L^2} &= \lim_{k \to \infty}\frac{\cE_{s_k}(u_{s_k},\psi)-\langle u_{s_k}, \psi \rangle_{L^2}}{s_k}\label{revision-1}\\
&= \lim_{k \to \infty}\langle u_{s_k}, \Gamma_{s_k} \psi \rangle_{L^2} 
= \langle u_0, \loglap \psi \rangle_{L^2} =
\cE_L(u_0,\psi). \nonumber
\end{align}
Here we used Theorem~\ref{teo-representation}(i) with $p=2$. Since also $\langle u_{s_k}, \psi \rangle_{L^2} \to \langle u_{0}, \psi \rangle_{L^2}$ as $k \to \infty$, and since we may choose $\psi \in C^2_c(\Omega)$ with $\langle u_{0}, \psi \rangle_{L^2} >0$, it follows from (\ref{revision-1}) that $\lambda^*>-\infty$ and
$$
\cE_L(u_0,\psi)= \lambda^* \langle u_0, \psi \rangle_{L^2} \qquad \text{for all $\psi \in \cH(\Omega)$.}
$$
Thus $u_0$ is an eigenfunction of $\loglap$ corresponding to the eigenvalue $\lambda^*$. Since $\lambda^* \le \lambda_1^L(\Omega)$, it follows from the definition of $\lambda_1^L(\Omega)$ that $\lambda^* = \lambda_1^L(\Omega)$. Since moreover $u_0 \ge 0$ and $\|u_0\|_{L^2}=1$, $u_0= \xi_1$ is the unique nonnegative $L^2$-normalized eigenfunction of $\loglap$. This contradicts (\ref{eq:contradiction-statement}), and hence (\ref{eq:eigenfunction-convergence}) is proved. Next, let $\lambda_*:= \liminf \limits_{s \to 0^+} \frac{ \lambda_1^s(\Omega)-1}{s}$, and let $s^k \to 0^+$ be a sequence such that $\frac{\lambda_1^{s^k}(\Omega)-1}{s^k} \to \lambda_*$. We have already seen that $u_{s^k} \to \xi_1$ in $L^2(\Omega)$. By the same argument as above,  we thus find that $\lambda_*>-\infty$ and that 
$$
\cE_L(\xi_1,\psi)= \lambda_* \langle \xi_1, \psi \rangle_{L^2} \qquad \text{for all $\psi \in \cH(\Omega)$}.
$$
Consequently, we have $\lambda_* = \lambda_1^L(\Omega)$, and together with (\ref{eq:limsup-est}) this implies
(\ref{eq:zero-deriv-lambda}). The proof is finished.
\end{proof}

We may now deduce the Faber-Krahn-inequality for the logarithmic Laplacian, which we restate here for the reader's convenience.

\begin{corollary}
\label{sec:faber-Krahn}
Let $\rho>0$. Among all bounded Lipschitz domains $\Omega$ with $|\Omega| = \rho$, the ball $B=B_r$ with $|B|=\rho$ minimizes $\lambda_1^L(\Omega)$.
\end{corollary}

\begin{proof}
Let $\Omega \subset \R^N$ be a bounded Lipschitz domain with $|\Omega|=\rho$. By \cite[Theorem 5]{BLMH}, we have
$$
\lambda_1^s(B) \le \lambda_1^s(\Omega) \qquad \text{for $s \in (0,1)$.}
$$
Consequently, we have
$$
\lambda_1^L(B) = \lim_{s \to 0^+} \frac{\lambda_1^s(B)-1}{s} \le \lim_{s \to 0^+} \frac{\lambda_1^s(\Omega)-1}{s}= \lambda_1^L(\Omega).
$$
\end{proof}

\begin{remark}
\label{sec:faber-krahn-remark}{\rm 
We note that we do not have a direct proof of Corollary~\ref{sec:faber-Krahn}, which is merely based on symmetrization arguments applied to the quadratic form $\cE_L$. Moreover, even though we have, with the notation above, $\lambda_1^s(B) < \lambda_1^s(\Omega)$ in the case where $\Omega \not = B$, the argument above does not imply that $\lambda_1^L(B) < \lambda_1^L(\Omega)$ when $\Omega \not = B$. The characterization of the equality case remains an open problem.
}
\end{remark}

\section{The maximum principle on bounded domains}
\label{sec:maximum-principle}

Let $\Omega \subset \R^N$ be a bounded domain. In this section, we present maximum principles for the operator $\loglap$ on $\Omega$. We start with a strong version of the maximum principle for pointwise solutions, which turns out to be a rather direct consequence of the representation (\ref{representation-regional}).

\begin{proposition} (Strong Maximum Principle for pointwise solutions)
\label{thm-strong-max-principle}
Let $\Omega \subset \R^N$ be a bounded domain with $h_\Omega +\rho_N \ge 0$ on $\Omega$. Moreover, let $u \in L^1_0(\R^N)$ be a function which is continuous on $\bar{\Omega}$, Dini continuous in $\Omega$, and satisfies
$$
\loglap u \ge 0 \quad \text{in $\Omega$,}\qquad u \ge 0 \qquad \text{on $\R^N \setminus \Omega$}.
$$
Then $u > 0$ in $\Omega$ or $u \equiv 0$ a.e. in $\R^N$.
\end{proposition}

\begin{proof}
Suppose by contradiction that $u$ is not positive in $\Omega$. Since $u$ is continuous on $\bar{\Omega}$, there exists a point $x_0 \in \Omega$ with
\begin{equation}
  \label{eq:x-0-eq}
u(x_0)= \min_{x \in \Omega} u(x)  \le 0,
\end{equation}
whereas, by (\ref{representation-regional}),
$$
0 \le [\loglap u](x_0) = c_N\int_{\Omega} \frac{ u(x_0) -u(y)}{|x-y|^{N} } dy  - c_N \int_{\R^N \setminus \Omega} \frac{u(y)}{|x-y|^{N}}\,dy + [h_\Omega(x_0)+\rho_N] u(x_0)\le  0,
$$
since all three terms are nonpositive by assumption and (\ref{eq:x-0-eq}). It follows that
\begin{equation}
  \label{eq:three-zero}
0= \int_{\Omega} \frac{ u(x_0) -u(y)}{|x-y|^{N} } dy  = \int_{\R^N \setminus \Omega} \frac{u(y)}{|x-y|^{N}}\,dy = [h_\Omega(x_0)+\rho_N] u(x_0).
\end{equation}
The first two equalities and (\ref{eq:x-0-eq}) imply that $u \equiv u(x_0) \le 0$ in $\Omega$ and $u \equiv 0$ a.e. on $\R^N \setminus \Omega$. Then the preceding arguments are still valid if we choose $x_0$ sufficiently close to the boundary so that $[h_\Omega(x_0)+\rho_N]>0$ by (\ref{eq:lower-estimate-kappa-omega}), which then implies that $u(x_0)=0$ by (\ref{eq:three-zero}). Hence $u \equiv 0$ in $\Omega$, and therefore $u \equiv 0$ \,a.e. in $\R^N$.
\end{proof}

In the following, we wish to extend the maximum principle to functions which satisfy the inequality $\loglap u \ge 0$ in weak sense. This also allows us to formulate necessary and sufficient spectral conditions on the validity of the maximum principle as given in Theorem~\ref{sec:maxim-princ-bound-characterization}. We need to introduce an appropriate space of functions $u$ which may have nonzero values outside of $\Omega$ and which allows to give a meaning to the inequality $\loglap u \ge 0$ in $\Omega$ in weak sense.
For this, we first let $\cW(\Omega)$ denote the space of all functions $u \in L^2(\Omega)$ such that
$$
{\mathfrak b}(u,\Omega):= \frac{1} 2 \int_{\Omega} \int_{\Omega}(u(x)-u(y))^2\k(x-y) dx dy <+\infty.
$$
Then $\cW(\Omega)$ is a Hilbert space with scalar product given by
$$
\langle u,v \rangle_{\cW(\Omega)}: = {\mathfrak b}(u,v,\Omega) + \langle u,v \rangle_{L^2(\Omega)},
$$
where
$$
{\mathfrak b}(u,v,\Omega):= \frac{1} 2 \int_{\Omega} \int_{\Omega}(u(x)-u(y))(v(x)-v(y))\k(x-y) dx dy.
$$
We also observe that, if $u \in \cW(\Omega)$ satisfies
$$
\int_{\Omega}u^2(x)\kappa_\Omega(x)\,dx < +\infty,
$$
then the trivial extension $\tilde u$ of $u$ to $\R^N$ is contained in $\cH(\Omega)$, and
\begin{equation}
  \label{eq:kappa-omega-representation}
\cE(\tilde u,\tilde u) = {\mathfrak b}(u,\Omega) + \int_{\Omega}u^2(x)\kappa_\Omega(x)\,dx.
\end{equation}
This follows immediately from (\ref{eq:alternative-representation-0}). In the following, we often identify $u$ with $\tilde u$.
We shall need the following analogue of Theorem~\ref{density} for the space $\cW(\Omega)$.

\begin{proposition}\label{density_1}
Let $\Omega \subset \R^N$ be a bounded Lipschitz domain. Then $C^\infty_c(\Omega)$ is dense in $\cW(\Omega)$. Moreover, if $u \in \cW(\Omega)$ is nonnegative, we have:
\begin{itemize}
\item[(i)] There exists a sequence of functions $u_n \in \cH(\Omega)$ with compact support in $\Omega$, $0 \le u_n \le u_{n+1} \le u$ for all $n \in \N$ and $u_n \to u$ in $\cW(\Omega)$.
\item[(ii)] There exists a sequence of nonnegative functions $u_n \in C^\infty_c(\Omega)$, $n \in \N$ with $u_n \to u$ in $\cW(\Omega)$.
\end{itemize}
\end{proposition}

\begin{proof}
Since for $u \in \cW(\Omega)$ we have $u^\pm \in \cW(\Omega)$, it suffices again to consider a nonnegative function $u \in \cW(\Omega)$. For $k \in \N$, let $u_k := \min \{u,k \}: \Omega \to \R$. Then we have that $0 \le u_k \le u_{k+1} \le u$ on $\Omega$
and
$$
|u_k(x)-u_k(y)| \le |u(x)-u(y)| \qquad \text{for all $x,y \in \Omega$.}
$$
From this it readily follows that $u_k \in \cW(\Omega)$. Moreover, for $v_k:= u-u_k=(u-k)^+$ we have that
$$
v_j(x) \le v_k(x) \quad \text{and}\quad |v_j(x)-v_j(y)| \le |v_k(x)-v_k(y)| \qquad \text{for $x,y \in \Omega$, $k<j$.}
$$
Since $v_k \to 0$ a.e. in $\Omega$, it follows by monotone convergence that
$$
\|u-u_k\|_{\cW(\Omega)}^2 =\|v_k\|^2_{\cW(\Omega)}= \int_{\Omega}(v_k(x)-v_k(y))^2 \k(x-y)\,dx dy + \int_\Omega v_k^2(x)\,dx \to 0
$$
as $k \to \infty$. Hence it suffices from now on to consider the case where $u$ is bounded,  which implies that $u \in \cW(\Omega) \cap L^2(\Omega,\log \frac{1}{\dist(x,\partial \Omega)})$. We extend $u$ on $\R^N$ by setting $u \equiv 0$ on $\R^N \setminus \Omega$. As remarked above, we then have that $u \in \cH(\Omega)$. Consequently, since $\|v\|_{\cW(\Omega)} \le  C \|v\|_{\cH(\Omega)}$ for every $v \in \cH(\Omega)$ with some constant $C>0$, the claim now follows from Theorem \ref{density}.
\end{proof}

\begin{remark}
\label{a-posteriori-remark}
{\rm Let $\Omega \subset \R^N$ be a bounded Lipschitz domain. From a logarithmic boundary Hardy inequality which we will prove in Proposition~\ref{log-boundary-harnack-lip-bound} in the Appendix, it follows that 
$$
\int_{\Omega} \kappa_\Omega(x) u^2(x)  \,dx \le C\Bigl( {\mathfrak b}(u,\Omega) + \|u\|_{L^2(\Omega)}^2\Bigr)\qquad \text{for all $u \in C^\infty_c(\Omega)$}    
$$
with a constant $C>0$, see Corollary~\ref{log-boundary-harnack-lip-bound-corollary} below. Since moreover $C^\infty_c(\Omega)$ is dense in $\cW(\Omega)$ by Proposition~\ref{density_1}, we deduce a posteriori from (\ref{eq:FKV-eq}) and (\ref{eq:kappa-omega-representation}) that the spaces $\cW(\Omega)$ and $\cH(\Omega)$ are isomorphic in this case -- with equivalent norms -- via the trivial extension map $\cW(\Omega) \to \cH(\Omega)$, $u \mapsto \tilde u$.
This fact may be of independent interest, but we will not use it in the paper since it does not simplify the proofs.}
\end{remark}

Next, we let $\cV(\Omega)$ denote the space of all measurable functions $u \in L^1_0(\R^N)$ such that the restriction of $u$ to $\Omega$ is contained in $\cW(\Omega)$. The following observation will allow us to define a weak notion of the inequality $\loglap u \ge 0$ for functions in $\cV(\Omega)$.

\begin{lemma}
\label{well-defined-lemma}
The quantities $\cE(u,v)$ and $\cE_L(u,v)$ are well defined for $u \in \cV(\Omega)$ and $v \in \cH(\Omega)$ with
$\supp v \subset \subset \Omega$.
\end{lemma}

\begin{proof}
Setting $\Omega':= \supp v$, we see that
	\begin{align*}
	&\frac{1}{2}\int_{\R^N} \int_{\R^N} |u(x)-u(y)|\cdot|v(x)-v(y)|\k(x-y) dxdy = \frac{1}{2} \int_{\Omega} \int_{\Omega} \dots + \int_{\Omega'} \int_{\R^N \setminus \Omega} \dots
\\
&\leq \bigl({\mathfrak b}(u,\Omega)\bigr)^{\frac{1}{2}} \bigl({\mathfrak b}(v,\Omega)\bigr)^{\frac{1}{2}} +\int_{\Omega'} |v(x)| \int_{\R^N\setminus \Omega} |u(x)-u(y)| \k(x-y) dydx,
	\end{align*}
where
\begin{align*}
\int_{\Omega'} |v(x)| &\int_{\R^N\setminus \Omega} |u(x)-u(y)|\k(x-y) dydx\\
&\le \int_{\Omega'}|v(x)| |u(x)| \kappa_{\Omega}(x)\,dx + \int_{\Omega'}
|v(x)| \int_{\R^N\setminus \Omega} |u(y)|\k(x-y) dydx \\
&\le c_1 \|v\|_{L^2(\Omega)} \|u\|_{L^2(\Omega)} + c_2 \|v\|_{L^1(\Omega)} \|u\|_{L^1_0}
\end{align*}
with
$$
c_1:= \sup_{x \in \Omega'} \kappa_\Omega(x)  \qquad \text{and}\qquad c_2: = \sup_{x \in \Omega',y \in \R^N \setminus \Omega} \k(x-y)(1+|y|)^N.
$$
Since $\Omega' \subset \subset \Omega$, the values $c_1$ and $c_2$ are finite. It thus follows that $\cE(u,v)$ is well-defined in Lebesgue sense with
$$
|\cE(u,v)| \le \bigl({\mathfrak b}(u,\Omega)\bigr)^{\frac{1}{2}} \bigl({\mathfrak b}(v,\Omega)\bigr)^{\frac{1}{2}} + c_1 \|v\|_{L^2(\Omega)} \|u\|_{L^2(\Omega)} + c_2 \|v\|_{L^1(\Omega)} \|u\|_{L^1_0}
$$
and constants $c_1,c_2$ depending only on $\Omega$ and $\supp v \subset \subset \Omega$.
Moreover, we have that
$$
\int_{\R^N}\int_{\R^N} |u(x)| |v(y)| \j(x-y)\,dx dy =  \int_{\Omega'}|v(y)| \int_{\R^N}|u(x)|  \j(x-y)\,dx dy
 \le c_3 \|v\|_{L^1(\Omega)} \|u\|_{L^1_0}
$$
with
$$
c_3 := \sup_{y \in \Omega',x \in \R^N} \j(x-y)(1+|x|)^N.
$$
Hence $\cE_L(u,v)$ is also well-defined with
$$
|\cE_L(u,v)| \le \bigl({\mathfrak b}(u,\Omega)\bigr)^{\frac{1}{2}} \bigl({\mathfrak b}(v,\Omega)\bigr)^{\frac{1}{2}} + (c_1+|\rho_N|) \|v\|_{L^2(\Omega)} \|u\|_{L^2(\Omega)} + (c_2+c_3) \|v\|_{L^1(\Omega)} \|u\|_{L^1_0}.
$$
The proof is complete.\end{proof}

We now recall the following definition from the introduction.

\begin{definition}
\label{maximum-principle-definition}
\begin{enumerate}
\item[i)] For a function $u \in \cV(\Omega)$, we say that $\loglap u \ge 0$ in $\Omega$ (in weak sense) if $\cE_L(u, \phi) \ge 0$ for all nonnegative $\phi \in C_c^\infty(\Omega)$.
\item[ii)] We say that $\loglap$ satisfies the maximum principle on $\Omega$ if for every $u \in \cV(\Omega)$ with
$\loglap u \ge 0$ in $\Omega$ and $u \ge 0$ in $\R^N \setminus \Omega$, we have $u \ge 0$ in $\R^N$.
\end{enumerate}
\end{definition}

\begin{remark}
\label{maximum-principle-remark}{\rm 
If $u \in L^1_0(\R^N)$ is uniformly Dini continuous on $\bar \Omega$, then $u \in \cV(\Omega)$. Moreover, if $\loglap u \ge 0$ in pointwise sense, then this also holds in weak sense. This follows since
 \begin{equation*}
\int_{\Omega} [\loglap u] \varphi\,dx = \cE_L(u,\varphi)
\end{equation*}
for functions $u \in L^1_0(\R^N)$ which are uniformly Dini continuous on $\bar \Omega$ and $\varphi \in C^\infty_c(\Omega)$. Indeed, in this case we have
$$
\int_{\Omega}|\varphi(x)| \int_{\R^N}|u(x)-u(y)|\k(x-y)\,dy dx \le \int_{\Omega}|\varphi(x)| \int_0^1 \frac{\omega_{u,\Omega}(r)}{r}\,dr \,dx  < +\infty,
$$
and therefore, by Fubini's theorem,
 \begin{align*}
\int_{\R^N}\varphi(x) \int_{\R^N}(u(x)-u(y))\k(x-y)&\,dy dx= \frac{1}{2}\Bigl (\int_{\R^N}\varphi(x) \int_{\R^N}(u(x)-u(y))\k(x-y)\,dy dx\\
 &+ \int_{\R^N}\varphi(y)\int_{\R^N}(u(y)-u(x))\k(x-y)\,dy dx\Bigr)= \cE(u,\varphi),
\end{align*}
which yields that
$$
\cE_L(u,\varphi)= \mathcal{E}(u,\varphi) - \int_{\R^N} \bigl( \j * u  - \rho_N u \bigr)\varphi\,dx  = \int_{\Omega} [\loglap u] \varphi\,dx.
$$
}
\end{remark}

In the proof of weak maximum principles, testing the equation in weak sense with $u^-$ is usually a key step. Related to this, we have the following useful property.

\begin{proposition}
\label{sec:maxim-princ-bound-u-}
Let $\Omega \subset \R^N$ be a bounded Lipschitz domain, and suppose that $u \in \cV(\Omega)$ satisfies $\loglap u \ge 0$ in $\Omega$ and $u \ge 0$ in $\R^N \setminus \Omega$. Then
$u^- \in \cH(\Omega)$ and
$$
\cE_L(u^-,u^-) \le 0.
$$
\end{proposition}

\begin{proof}
We first note that
\begin{equation}
  \label{eq:c-infty-to-compact-support-H}
\text{$\cE_L(u,w) \ge 0$ for all nonnegative $w \in \cH(\Omega)$ with $\supp w \subset \subset \Omega$.}
\end{equation}
Indeed, for every such $w$, Theorem~\ref{density}(ii) yields a sequence of nonnegative functions $w_n \in C^\infty_c(\Omega)$ with $w_n \to w$ in $\cH(\Omega)$. Moreover, an inspection of the proof of Theorem~\ref{density}(ii) shows that there exists a compact set $K \subset \Omega$ with $\supp w_n \subset K$ for all $n \in \N$. Hence the estimates in the proof of Lemma~\ref{well-defined-lemma} show that
$$
\cE_L(u,w) = \lim_{n \to \infty}\cE_L(u,w_n) \ge 0,
$$
as claimed in (\ref{eq:c-infty-to-compact-support-H}).

Next we note that $u^- \in \cW(\Omega)$ since $u \in \cW(\Omega)$. Hence, by Proposition~\ref{density_1}(i),
 there exists a sequence of nonnegative functions $v_n \in \cH(\Omega)$ with compact support in $\Omega$,
  $0 \le v_n \le v_{n+1} \le u^-$ in $\Omega$ and $v_n  \to u^-$ in $\cW(\Omega)$ as $n \to \infty$. Then
$$
\cE(u,v_n) = {\mathfrak b}(u,v_n,\Omega)+ \int_{\Omega}v_n(x) \int_{\R^N \setminus \Omega}(u(x)-u(y))\k(x-y) dxdy,
$$
where, as $n \to \infty$,
\begin{align*}
{\mathfrak b}(u,v_n,\Omega)\to {\mathfrak b}(u,u^-,\Omega) &= \frac{1}{2}\int_{\Omega} \int_{\Omega}(u(x)-u(y)) \cdot (u^-(x)-u^-(y))\k(x-y)\ dxdy\\
& \le - \frac{1}{2}\int_{\Omega} \int_{\Omega}(u^-(x)-u^-(y))^2\k(x-y)\ dxdy = -{\mathfrak b}(u^-,\Omega)
\end{align*}
and, since $0 \le v_n \le u^-$ in $\Omega$ and $u \ge 0$ in $\R^N \setminus \Omega$,
\begin{align*}
\int_{\Omega} v_n(x) \int_{\R^N \setminus \Omega}(u(x)-u(y))\k(x-y) dy dx
& \le \int_{\Omega} v_n(x)u(x) \int_{\R^N \setminus \Omega}\k(x-y) dy dx \\
& = -\int_{\Omega} v_n(x)u^-(x) \kappa_{\Omega}(x)\,dx.
\end{align*}
With the help of (\ref{eq:c-infty-to-compact-support-H}), we thus infer that
\begin{align*}
0 &\le \int_{\Omega} v_n(x)u^-(x) \kappa_{\Omega}(x)\,dx  \le -\int_{\Omega} v_n(x) \int_{\R^N \setminus \Omega}(u(x)-u(y))\k(x-y) dy dx \\
&= {\mathfrak b}(u,v_n,\Omega) - \cE(u,v_n) = {\mathfrak b}(u,v_n,\Omega) - \cE_L(u,v_n) - \int_{\R^N} \bigl( \j * u  -\rho_N u \bigr)v_n\,dx\\
& \le {\mathfrak b}(u,v_n,\Omega) - \int_{\R^N} \bigl( \j * u  -\rho_N u \bigr)v_n\,dx\\
&\le  o(1)- {\mathfrak b}(u^-,\Omega) - \int_{\R^N} \bigl( \j * u  -\rho_N u \bigr)u^-\,dx\\
&=  o(1)- {\mathfrak b}(u^-,\Omega) - \int_{\R^N} [\j * u]u^-\,dx - \rho_N \|u^-\|_{L^2(\Omega)}^2\\
&\le  o(1)- {\mathfrak b}(u^-,\Omega) +\int_{\R^N} [\j * u^-]u^-\,dx - \rho_N \|u^-\|_{L^2(\Omega)}^2 \qquad \text{as $n \to \infty$.}
\end{align*}
By monotone convergence, we thus conclude that
$$
\int_{\Omega} [u^-(x)]^2 \kappa_{\Omega}(x)\,dx  \le - {\mathfrak b}(u^-,\Omega) +\int_{\R^N} [\j * u^-]u^-\,dx - \rho_N \|u^-\|_{L^2(\Omega)}^2,
$$
and therefore $u^- \in \cH(\Omega)$ with
$$
\cE_L(u^-,u^-) = {\mathfrak b}(u^-,\Omega)+ \int_{\Omega} [u^-(x)]^2 \kappa_{\Omega}(x)\,dx
-\int_{\R^N} [\j * u^-]u^-\,dx + \rho_N \|u^-\|_{L^2(\Omega)}^2 \le 0.
$$
The proof is complete.\end{proof}

We may now complete the proof of Theorem~\ref{sec:maxim-princ-bound-characterization}, which we restate here for the reader's convenience.
\begin{theorem}
 \label{sec:maxim-princ-bound-characterization-section}
Let $\Omega \subset \R^N$ be a bounded Lipschitz domain. Then $\loglap$ satisfies the maximum principle on $\Omega$ if and only if $\lambda_1^L(\Omega)>0$.
\end{theorem}

\begin{proof}
Suppose first that $\lambda_1^L(\Omega)>0$, and let $u \in \cV(\Omega)$ satisfy $\loglap u \ge 0$ in $\Omega$ and $u \ge 0$ in $\R^N \setminus \Omega$. By Proposition~\ref{sec:maxim-princ-bound-u-}, we then have $u^- \in \cH(\Omega)$ and
$$
\lambda_1^L(\Omega)\|u^-\|^2_{L^2(\Omega)} \le \cE_{L}(u^-,u^-) \le 0.
$$
Consequently, $u^- \equiv 0$, and therefore $u \ge 0$ in $\R^N$. Thus $\loglap$ satisfies the maximum principle on $\Omega$. Suppose now that $\lambda_1^L(\Omega)\le 0$, and let $\xi_1 \in \cH(\Omega) \subset \cV(\Omega)$ be the unique positive $L^2$-normalized eigenfunction of $\loglap$.
With $u:= -\xi_1$, we then have $u \equiv 0$ in $\R^N \setminus \Omega$, $u<0$ in $\Omega$ and $\loglap u = \lambda_1^L(\Omega)u \ge 0$ in $\Omega$.
Hence $\loglap$ does not satisfy the maximum principle on $\Omega$.
\end{proof}

It is now easy to deduce Corollary~\ref{non-max-principle-simple}, which again we restate for the reader's convenience.

\begin{corollary}
\label{sec:proposition-max-principle-spectral}
Let $\Omega \subset \R^N$ be a bounded Lipschitz domain. Then we have
$\lambda_1^L(\Omega) \le \log \lambda_1^1(\Omega)$,  where $\lambda_1^1(\Omega)$ is the first eigenvalue of the Dirichlet Laplacian $-\Delta$ on $\Omega$. Moreover, if $\lambda_1^1(\Omega) \le 1$, then $\loglap$ does not satisfy the maximum principle on $\Omega$.
\end{corollary}

\begin{proof}
By \cite[Corollary 4]{musina-nazarov}, we have
$$
\lambda_1^s(\Omega) < \bigl[\lambda_1^1(\Omega)\bigr]^s \qquad \text{for $s \in (0,1)$.}
$$
Consequently,
$$
\lambda_1^L(\Omega) = \lim_{s \to 0^+} \frac{\lambda_1^s(\Omega)-1}{s}\le \lim_{s \to 0^+} \frac{\bigl[\lambda_1^1(\Omega)\bigr]^s-1}{s}
= \log \lambda_1^1(\Omega).
$$
If $\lambda_1^1(\Omega) \le 1$, we conclude that $\lambda_1^L(\Omega) \le 0$, and hence $\loglap$ does not satisfy the maximum principle on $\Omega$ by Theorem~\ref{sec:maxim-princ-bound-characterization}.
\end{proof}

The following corollary is a variant of Proposition~\ref{thm-strong-max-principle}, which applies to functions $u \in \cV(\Omega)$ satisfying the inequality $\loglap u \ge 0$ in weak sense.

\begin{corollary}
\label{sec:corollary-max-principle-h-function}
Let $\Omega \subset \R^N$ be a bounded Lipschitz domain, and suppose that $h_\Omega + \rho_N \ge 0$ on $\Omega$. Then $\lambda_1^L(\Omega) >0$, and thus $\loglap$ satisfies the maximum principle on $\Omega$.
\end{corollary}

\begin{proof}
By definition of $\lambda_1^L(\Omega)$ and (\ref{eq:alternative-representation}), it follows that $\lambda_1^L(\Omega) \ge 0$.
We claim that $\lambda_1^L(\Omega)>0$. Assuming by contradiction that $\lambda_1^L(\Omega)=0$, we let, as before, $\xi_1 \in \cH(\Omega)$ be the corresponding nonnegative $L^2$-normalized eigenfunction. Then, by (\ref{eq:alternative-representation}),
$$
0 = \lambda_1^L(\Omega)\|\xi_1\|_{L^2(\Omega)}^2 = \cE_L(\xi_1,\xi_1) = {\mathfrak b}(\xi_1,\Omega) +  \int_\Omega [h_{\Omega}(x)+\rho_N] \xi_1^2(x) \,dx \ge {\mathfrak b}(\xi_1,\Omega).
$$
From the definition of ${\mathfrak b}(\xi_1,\Omega)$, we then deduce that $\xi_1 \equiv c$ in $\Omega$ with a constant $c>0$. Then the above estimate gives
$$
0 \ge  c^2 \int_\Omega [h_{\Omega}(x)+\rho_N] \,dx
$$
and therefore $h_\Omega + \rho_N \equiv 0$ in $\Omega$ by assumption. This however contradicts the fact that $h_\Omega(x) \to \infty$ as $\dist(x,\partial \Omega) \to 0$. The proof is finished.
\end{proof}

Theorem~\ref{thm-strong-max-principle} and Corollary~\ref{sec:corollary-max-principle-h-function} show that it is useful to have estimates for $h_\Omega$ from below. In the following lemma, we consider the special case $\Omega = B_r$ for some $r>0$.

\begin{lemma}
\label{sec:maxim-princ-h-lower-bounds}
Let $\Omega = B_r$ for some $r>0$. Then we have
$$
h_\Omega(x) \ge 2 \log \frac{1}{r}   \qquad \text{for $x \in \Omega$,}
$$
and equality is attained at $x = 0$.
Consequently, $h_\Omega + \rho_N \ge 0$ on $\Omega$ if and only if
$$
r \le r_N:= 2 e^{\frac{1}{2}( \psi(\frac{N}{2})-\gamma)}.
$$
\end{lemma}

\begin{proof}
We first show that $h_\Omega(x) \ge h_\Omega(0)$ for $x \in \Omega$. Indeed, let $x \in \Omega$ and $r_1(x) = \Bigl(\frac{|\Omega\setminus B_1(x)|}{|B_1|}+1\Bigr)^{\frac{1}{N}}$, so that $|\Omega\setminus B_1(x)|= |B_{r_1(x)} \setminus B_1|$. It is then easy to see that
\begin{align*}
 \int_{ \Omega\setminus B_1(x)}\frac1{|x-y|^N} dy &=
 \int_{ [\Omega-x] \setminus B_1}\frac1{|z|^N} dz \le
 \int_{B_{r_1(x)} \setminus B_1}\frac1{|z|^N} dz= |\mathcal{S}^{N-1}| \log r_1(x)\\
&= |B_1| \log \Bigl(\frac{|\Omega \setminus B_1(x)|}{|B_1|}+1\Bigr)= |B_1| \log \frac{|B_1| + |\Omega \setminus B_1(x)|}{|B_1|}.
\end{align*}
Moreover, we choose $r_2(x) := \Bigl(1- \frac{|B_1(x) \setminus \Omega|}{|B_1|}\Bigr)^{\frac{1}{N}}$, so that $|B_1(x) \setminus \Omega| = |B_1 \setminus  B_{r_2(x)}|$. We then have, for $x \in \Omega$,
\begin{align*}
\int_{B_1(x) \setminus \Omega}\frac1{|x-y|^N} dy  &=
 \int_{B_1 \setminus  [\Omega-x]}\frac1{|z|^N} dz \ge
 \int_{B_1 \setminus  B_{r_2(x)}}\frac1{|z|^N} dz = - |\mathcal{S}^{N-1}| \log r_2 \\
&= - |B_1|\log \Bigl(1- \frac{|B_1(x) \setminus \Omega|}{|B_1|}\Bigr)= |B_1| \log \frac{|B_1|}{|B_1|- |B_1(x) \setminus \Omega|}.
\end{align*}
Consequently,
$$
h_{\Omega}(x)  \ge c_N |B_1| \Bigl( \log \frac{|B_1|}{|B_1|- |B_1(x) \setminus \Omega|}  - \log \frac{|B_1| + |\Omega \setminus B_1(x)|}{|B_1|}\Bigr)=: \tilde h_{\Omega}(x).
$$
Now the quantity $\tilde h_\Omega$ is minimized at points $x \in \Omega$ where $|B_1(x) \setminus \Omega|$ is minimal.
Thus, in the case where $\Omega =B_r$ for some $r>0$, then $\tilde h_\Omega$ is minimized at $x=0$, and we have
$\tilde h_\Omega(0)= h_\Omega(0)$ since all inequalities in the above estimate are equalities in the case $x=0$.
Moreover, we compute that
$$
h_\Omega(0)= - c_N |\mathcal{S}^{N-1}| \log r = c_N |\mathcal{S}^{N-1}|\log \frac{1}{r} = 2 \log \frac{1}{r}
$$
in this case.
\end{proof}

\begin{corollary}
\label{sec:maxim-princ-bound}
Suppose that $\Omega \subset \R^N$ is a bounded Lipschitz domain with
\begin{equation}
  \label{eq:cond-max-principle}
\frac{|\Omega|}{|B_1|} \le 2^N e^{\frac{N}{2}(\psi(\frac{N}{2})-\gamma)}.
\end{equation}
Then $\loglap$ satisfies the maximum principle on $\Omega$.
\end{corollary}

\begin{proof}
Let $r:= \Bigl(\frac{|\Omega|}{|B_1|}\Bigr)^{\frac{1}{N}}$, so that $|\Omega| = |B_r|$. By assumption (\ref{eq:cond-max-principle}) we have $r \le r_N$, which implies that $\lambda_1^L(B_r) > 0$ by Theorem~\ref{sec:maxim-princ-bound-characterization-section}, Corollary~\ref{sec:corollary-max-principle-h-function} and Lemma~\ref{sec:maxim-princ-h-lower-bounds}.

By Corollary~\ref{sec:faber-Krahn}, it then follows that $\lambda_1^L(\Omega)>0$, and thus $\loglap$ satisfies the maximum principle on $\Omega$ by Theorem~\ref{sec:maxim-princ-bound-characterization-section}.
\end{proof}

Now Corollary~\ref{sec:main-result-max-principle} readily follows from Corollaries~\ref{sec:corollary-max-principle-h-function} and~\ref{sec:maxim-princ-bound}.

\section{Regularity and boundary decay}
\label{sec:regul-bound-decay}

This section is devoted to the proof of Theorem~\ref{pr 2.1-main}. In the following, we fix a function $\ell: (0,\infty) \to (0,\infty)$ with
\begin{equation}
  \label{eq:properties-ell}
\ell(t) \to 0 \quad \text{and}\quad \frac{\log \ell(t) }{\Bigl(-\log t\Bigr)^\sigma} \to 0 \qquad \text{as $t \to 0\quad$ for any $\sigma>0$.}
\end{equation}
A possible choice for $\ell$ is a positive extension of the function $t \mapsto -\frac{1}{\log t}$, $t \in (0,\frac{1}{2})$. We also define the constant
\begin{equation}
  \label{eq:def-kappa-N}
\kappa_N:= \int_{\R^{N-1}} \frac{1}{\bigl(1+ |z'|^2\bigr)^{\frac{N}{2}}}\,d z' \quad \text{for $N \ge 2$,}\qquad \kappa_1:= 1.
\end{equation}

We need the following integral estimates.

\begin{lemma}
\label{estimate-h-t}
Let the function $h: [0,\infty) \setminus \{1\} \to (0,\infty)$ be defined by

$$
h(t):= \int_{S^{N-1}}|t e_N-y|^{-N}d\sigma(y).
$$
Then $h(t)=  \frac{h_0(t)}{|t-1|}$ for $t \in [0,\infty) \setminus \{1\}$ with a bounded continuous function $h_0: [0,\infty) \to \R$ satisfying $h_0(1)= \kappa_N$.
\end{lemma}

\begin{proof}
For $\tau>0$, we consider the rescaled and translated sphere $S_\tau:=\tau(S^{N-1}-e_N)$ with radius $\tau$ which touches the origin and is contained in the half space $\{x_N \le 0\}$. For $t \in [0,\infty) \setminus \{1\}$ we then have, with $\tau_t:=\frac{1}{|t-1|}$,
$$
|t-1| h(t)= |t-1|^{N} \int_{S_{\tau_t}}\bigl|t e_N-\bigl(|t-1| y +e_N\bigr)|^{-N}d\sigma(y)= \int_{S_{\tau_t}}\Bigl|\frac{t-1}{|t-1|} e_N- y\Bigr|^{-N}d\sigma(y) =: h_0(t).
$$
Observe that this defines a continuous function $h_0: [0,\infty) \setminus \{1\} \to \R$ which is bounded on $[0,\infty) \setminus (1-\eps,1+\eps)$ for any $\eps \in (0,1)$. It thus remains to show that $h_0(t) \to \kappa_N$ as $t \to 1$, i.e., that
\begin{equation}
  \label{eq:remains-to-show-sphere-integral}
\int_{S_{\tau}}\bigl|\pm e_N- y\bigr|^{-N}d\sigma(y) \to \kappa_N \qquad \text{as $\tau \to \infty$.}
\end{equation}
For this we choose $\frac{N-1}{N} < \alpha <1$, and we let $\tau>1$. Moreover, we consider the local graph parametrization
$$
\psi_\tau:  B_{\tau^\alpha} \to S_{\tau},\qquad \psi_\tau(z')= (z',\phi_\tau(z'))\quad \text{with $\phi_\tau(z')=\sqrt{\tau^2-|z'|^2}-\tau$},
$$
where $B_{\tau^\alpha}:= B_{\tau^\alpha}(0) \subset \R^{N-1}$, and we let $A_\tau:= S_\tau \setminus \psi_\tau(B_{\tau^\alpha})$. Then we have $|y \pm e_N| \ge \tau^\alpha$ for every $y \in A_\tau$ and therefore
\begin{equation}
  \label{eq:sphere-integral-first}
\int_{A_\tau}|e_N\pm y |^{-N}d\sigma(y) \le |S_\tau|\tau^{-\alpha N} = O(\tau^{N-1-\alpha N}) \to 0 \qquad \text{as $\tau \to \infty.$}
\end{equation}
Moreover,
$$
\int_{\psi_\tau(B_{\tau^\alpha})}|e_N\pm y |^{-N}d\sigma(y) = \int_{B_{\tau^\alpha}}h_\tau(z')dz'
\quad \text{with}\quad h_\tau^\pm(z'):= \frac{\tau }{\sqrt{\tau^2-|z'|^2}}\bigl(|z'|^2 + [\phi_\tau(z')\pm 1]^2 \bigr)^{-\frac{N}{2}}.
$$
Note that $0 \le \tau-\sqrt{\tau^2-|z'|^2}\le 1- \sqrt{1-|z'|^2} \le 1$ for $|z'|\le 1$, $\tau\ge 1$, which implies that
$$
[\phi_\tau(z')\pm 1]^2 = \Bigl(\pm 1- \bigl(\tau-\sqrt{\tau^2-|z'|^2}\,\bigr)\Bigr)^2 \ge 1-|z'|^2 \qquad \text{for $|z'|\le 1$, $\tau\ge 1$.}
$$
Furthermore, $\frac{\tau }{\sqrt{\tau^2-|z'|^2}} \le \frac{\tau}{\sqrt{\tau^2-\tau^{2\alpha}}} \le 2$ for $z' \in B_{\tau^\alpha}$ if $\tau \ge \bigl( \frac{4}{3}\bigr)^{\frac{1}{2-2\alpha}}$. Consequently,
\begin{equation}
|h_\tau(z')|\le  2\Bigl(|z'|^2 + [\phi_\tau(z')\pm 1]^2 \Bigr)^{-\frac{N}{2}} \le 2\Bigl( 1_{\{|z'| \ge 1\}} |z'|^{-N} + 1_{\{|z'| \le 1\}}\Bigr) \label{dom-conv-1}
\end{equation}
for $\tau \ge \bigl( \frac{4}{3}\bigr)^{\frac{1}{2-2\alpha}}$, $|z'| \le B_{\tau^\alpha}$, where the RHS of (\ref{dom-conv-1}) defines an integrable function on $\R^{N-1}$.
Furthermore, we see that $h_\tau(z') \to \bigl(1+ |z'|^2\bigr)^{-\frac{N}{2}}$ as $\tau \to \infty$ for every $z' \in \R^{N-1}$. Combining this pointwise convergence property with (\ref{dom-conv-1}) and Lebesgue's theorem, we find that
\begin{equation}
  \label{eq:sphere-integral-second}
\int_{\psi_\tau(B_{\tau^\alpha})}|e_N \pm y |^{-N}d\sigma(y) \to \int_{\R^{N-1}} \bigl(1+ |z'|^2\bigr)^{-\frac{N}{2}}\,d z'  = \kappa_N \qquad \text{as $\tau \to \infty$.}
\end{equation}
Now (\ref{eq:sphere-integral-first}) and (\ref{eq:sphere-integral-second}) yield (\ref{eq:remains-to-show-sphere-integral}), as required.
\end{proof}

\begin{lemma}
\label{estimate-kappa-Omega-ball}
Let $R>0$  and   $\rho(x)= {\rm dist}(x,\partial B_R)= |x|-R$ for $x \in \R^N \setminus B_R$. Then we have that
\begin{equation}
    \label{eq:kappa-omega-C-2-est}
\int_{B_R}|x-y|^{-N}\,dy \ge  -(\kappa_N + o(1)) \log \rho(x)  \qquad \text{for $x \in \R^N \setminus B_R$ as $\rho(x) \to 0$.}
  \end{equation}
\end{lemma}

\begin{proof}
For $x \in \R^N \setminus B_R$, we have that
\begin{align}
&\int_{B_R}|x-y|^{-N}\,dy = \int_{0}^{R} \tau^{N-1} \int_{S^{N-1}} |x-\tau y|^{-N}\,dy = \int_{0}^{R} \tau^{-1} \int_{S^{N-1}} |\frac{x}{\tau}- y|^{-N}\,dy \nonumber\\
&= \int_{0}^{R}  \tau^{-1} \frac{h_0(\bigl|\frac{x}{\tau}\bigr|)}{| |\frac{x}{\tau}|-1|}\,d\tau =\int_{0}^{R-\ell(\rho(x))} \frac{h_0(\bigl|\frac{x}{\tau}\bigr|)}{|x|-\tau}\,d\tau + \int_{R-\ell(\rho(x))}^{R} \frac{h_0(\bigl|\frac{x}{\tau}\bigr|)}{|x|-\tau}\,d\tau. \label{estimate-kappa-Omega-ball-proof-1}
\end{align}
Since $\ell(\rho(x)) \to 0$ as $|x| \to R$, we have
$$
\sup_{R-\ell(\rho(x)) \le \tau \le R}\:\: \Bigl|  h_0(\bigl|\frac{x}{\tau}\bigr|)-\kappa_N \Bigr|  \to 0 \qquad \text{as $|x| \to R$}
$$
by Lemma~\ref{estimate-h-t} and therefore, by \eqref{eq:properties-ell},
\begin{align}
\int_{R-\ell(\rho(x))}^{R} &\frac{h_0(\bigl|\frac{x}{\tau}\bigr|)}{|x|-\tau}\,d\tau= \bigl(\kappa_N + o(1)\bigr)\int_{R-\ell(\rho(x))}^{R}\frac{1}{|x|-\tau}d\tau = \bigl(\kappa_N + o(1)\bigr)\int_{\rho(x)}^{\rho(x)+\ell(\rho(x))}\frac{1}{\tau}d\tau \nonumber\\
&= \bigl(\kappa_N + o(1)\bigr) \bigl(\log \bigl(\rho(x)+\ell(\rho(x))\bigr)- \log(\rho(x))\bigr)\ge \bigl(\kappa_N + o(1)\bigr)
\bigl(\log \bigl(\ell(\rho(x))\bigr)- \log(\rho(x))\bigr) \nonumber\\
&=\bigl(\kappa_N + o(1)\bigr) \bigl(- \log(\rho(x))\bigr) \quad\ \text{as \,$\rho(x) \to 0$.} \label{estimate-kappa-Omega-ball-proof-2}
\end{align}
Moreover, since $h_0$ is a bounded function,
\begin{align}
&\Big|\int_0^{R-\ell(\rho(x))} \frac{h_0(\bigl|\frac{x}{\tau}\bigr|)}{|x|-\tau}\,d\tau\Bigr| \le \|h_0\|_{L^\infty}
\int_{\rho(x)+\ell(\rho(x))}^{|x|} \frac{1}{\tau}d\tau \nonumber \\
&= \|h_0\|_{L^\infty} \bigl(\log (|x|) - \log \bigl(\rho(x)+\ell(\rho(x))\bigr)\bigr)\le
\|h_0\|_{L^\infty} \bigl(\log (|x|) - \log \bigl(\ell(\rho(x))\bigr)\bigr) \nonumber \\
&= o(1)\bigl(- \log(\rho(x))\bigr) \quad\ \text{as \,$\rho(x)= |x|-R \to 0$.}
\label{estimate-kappa-Omega-ball-proof-3}
\end{align}
Combining (\ref{estimate-kappa-Omega-ball-proof-1}), (\ref{estimate-kappa-Omega-ball-proof-2}) and~(\ref{estimate-kappa-Omega-ball-proof-3}), we get (\ref{eq:kappa-omega-C-2-est}).
\end{proof}

Next, we construct a barrier function which will allow us to prove Theorem~\ref{pr 2.1-main} with the help of the weak maximum principle.

\begin{lemma}
\label{sec:regul-bound-decay-lemma-radial}
Let $\tau \in (0,\frac{1}{2})$ and $0<R<\frac{1}{2}$. Then there exists $\delta>0$ and a continuous function $V \in L^1_0(\R^N)$ with the following properties:
\begin{itemize}
\item[(i)] $V \equiv 0$ in $B_R$ and $V>0$ in $\R^N \setminus \overline{B_R}$;
\item[(ii)] $V$ is (locally) Dini continuous in $\R^N \setminus \overline{B_R}$;
\item[(iii)] $V(x)= \bigl(-\ln \rho(x)\bigr)^{-\tau}$ for $x \in A_{R,\delta}$;
\item[(iv)] $\loglap V(x) \ge \frac{(1-2\tau)\kappa_N c_N}{2(1-\tau)} \bigl(-\ln \rho(x)\bigr)^{1-\tau}$ for $x \in A_{R,\delta}$.
\end{itemize}
Here $c_N= \pi^{- \frac{N}{2}}  \Gamma(\frac{N}{2})$ as before, $\kappa_N$ is given in (\ref{eq:def-kappa-N}), $A_{R,\delta}:= \{x \in \R^N \:: R<|x| <R+\delta\}$ and
$\rho(x):= |x|-R= \dist(x,B_R)$ for $x \in \R^N$ with $|x|>R$.
\end{lemma}

\begin{proof}
We consider the function $r \mapsto \bigl(-\log r\bigr)^{-\tau}$ which is positive, strictly increasing and concave in the interval $(0,\delta_\tau]$ with $\delta_\tau:= e^{-(\tau+1)}$. Indeed, this follows since
$$
\frac{d}{dr}\frac{1}{( -\log r)^{\tau}}  = \frac{\tau}{r (-\log r)^{\tau+1}}\qquad \text{and}\qquad  \frac{d^2}{dr^2} \frac{1}{
(-\log r)^{\tau}}  = -
\frac{\tau}{r^2 (-\log r)^{\tau+1} }\Bigl(1 +\frac{\tau+1}{\log r} \Bigr).
$$
Let $g:(0,\infty) \to \R$ be a positive extension of class $C^1$ of the function $(0,\delta_\tau] \to \R$, $r \mapsto \bigl(-\log r\bigr)^{-\tau}$ such that $g$ decays to zero exponentially as $r \to \infty$.
In particular, we have
\begin{equation}
  \label{eq:g-a-b-estimate}
g(b)-g(a) \le \min \{g(b-a),g'(a)(b-a)\} \qquad \text{for \,$0 < a \le b < \delta_\tau$.}
\end{equation}
We now let $R>0$ and consider the radial function
$$
V \in L^1_0(\R^N),\qquad V(x)=
\left \{
  \begin{aligned}
  &0&&\quad \text{for \, $|x| \le R$},\\
  &g(\rho(x))&&\quad \text{for \, $|x| > R$}.
  \end{aligned}
\right.
$$
Then properties (i), (ii) and (iii) are obviously satisfied with $\delta = \delta_\tau$. It remains to show that (iv) holds after making $\delta$ smaller if necessary. To see this, we consider $x \in A_{R,\delta_\tau}$ from now on, and we write
$$
\loglap V(x)= c_N \Bigl(I_1(x)+ I_2(x)+ I_3(x)\Bigr)
$$
with
$$
I_1(x)=  \int_{B_1(x) \setminus B_R} \frac{ V(x) -V(y)}{|x-y|^{N} } dy,\qquad I_2(x) = -  [\j * V](x) = -  \int_{\R^N \setminus B_1(x)} \frac{V(y)}{|x-y|^{N}}\,dy
$$
and
$$
I_3(x)= \Bigl[\rho_N + \int_{B_1(x) \cap B_R}\frac{1}{|x-y|^N}dy\Bigr]V(x),
$$
By Lemma~\ref{estimate-kappa-Omega-ball} and since $0<R<\frac{1}{2}$, we have
$$
\int_{B_1(x) \cap B_R}\frac{1}{|x-y|^N}dy = \int_{B_R}\frac{1}{|x-y|^N}dy = -(\kappa_N+o(1)) \log \rho(x) \quad \  \text{as \,$\rho(x) \to 0$,}
$$
which implies that
\begin{equation}
  \label{eq:I-3-condition}
I_3(x) =   (\kappa_N + o(1))(-\log \rho(x))^{1-\tau} \quad\  \text{as \,$\rho(x) \to 0$.}
\end{equation}
Moreover, by a standard estimate as in the proof of Lemma~\ref{continuity-convolution}, there exists $C>0$ such that
\begin{equation}
  \label{eq:I-2-condition}
|I_2(x)| \le C\|V\|_{L^1_0} \qquad \text{for any $x \in A_{R,\delta_\tau}$.}
\end{equation}

In view of (\ref{eq:I-3-condition}) and (\ref{eq:I-2-condition}), property (iv) follows for $\delta>0$ sufficiently small once we have shown that
\begin{equation}
\label{eq:sufficient-I-1-cond}
-I_1(x) \le \Bigl(\frac{ \tau \kappa_N}{1-\tau}  + o(1)\Bigr)(-\log \rho(x))^{1-\tau} \quad\ \text{as \,$\rho(x) \to 0$.}
\end{equation}
Using the definition of $V$ and the fact that $g$ is increasing on $(0,\delta_\tau)$, we find that
$$
-I_1(x) \le \int_{\text{\tiny $|x|\! \le \!|y|\! \le \! |x|+1$}}\frac{V(y)- V(x)}{|x-y|^{N} } dy =I_{1,1}(x)+I_{1,2}(x) + I_{1,3}(x)
$$
with
$$
I_{1,1}(x) = \int_{\text{\tiny $|x|\! \le\! |y|\! \le\! |x|+\rho(x)$}}\frac{V(y)- V(x)}{|x-y|^{N} } dy,
\quad I_{1,2}(x)= \int_{\text{\tiny $|x|\!+\!\rho(x)\!\le \!|y|\! \le \!|x|\!+ \!\ell(\rho(x))$}}\frac{V(y)- V(x)}{|x-y|^{N} } dy
$$
and
$$
I_{1,3}(x) = \int_{\text{\tiny $|x|\!+ \!\ell(\rho(x))\! \le \!|y| \!\le \!|x|+1$}}\frac{V(y)- V(x)}{|x-y|^{N} } dy,
$$
where, as before, the function $\ell$ is given with the properties (\ref{eq:properties-ell}). We compute that
\begin{align*}
I_{1,2}(x) &= \int_{\text{\tiny $|x|\!+\!\rho(x)\! \le \!|y|\! \le \!|x|\!+\!\ell(\rho(x))$}} \frac{g(|y|-R)-g(|x|-R)}{|x-y|^N}dy\\
&= \int_{|x|+\rho(x)}^{|x|+\ell(\rho(x))}\Bigl(g(t-R)- g(|x|-R)\Bigr)t^{N-1} \int_{S^{N-1}}|x-t \theta|^{-N}\,d\theta d t,
\end{align*}
where, by Lemma~\ref{estimate-h-t},
\begin{equation}
\label{eq:sphere-est}
t^{N-1} \int_{S^{N-1}}|x-t \theta|^{-N}\,d\theta =  \frac{1}{t}\int_{S^{N-1}}|\frac{x}{t}- \theta|^{-N}\,d\theta=\frac{h_0(|\frac{x}{t}|)}{t \bigl||\frac{x}{t}|-1\bigr|}=\frac{h_0(|\frac{x}{t}|)}{t-|x|}
\end{equation}
for $t>|x|$. Since $|x| \to R$ and $\ell(\rho(x)) \to 0$ as $\rho(x) \to 0$, we have
$$
\sup_{\text{\tiny $|x|\!+\!\rho(x)\! \le \!t\! \le \!|x|+\ell(\rho(x))$}}\Bigl| h_0(|\frac{x}{t}|)-\kappa_N\Bigr| \to 0
\quad\ \text{as \,$\rho(x) \to 0$},
$$
by Lemma~\ref{estimate-h-t} and therefore,
\begin{align*}
I_{1,2}(x)&=(\kappa_N + o(1)) \int_{|x|+\rho(x)}^{|x|+\ell(\rho(x))}\frac{g(t-R)-g(|x|-R)}{t-|x|}\,dt\\
&=(\kappa_N + o(1))  \Bigl(\int_{\rho(x)}^{\ell(\rho(x))}\frac{g(t+\rho(x))}{t}\,dt - g(\rho(x)) \Bigl( \ln \ell(\rho(x)) -\ln \rho(x)\Bigr)\Bigr)\\
&\le (\kappa_N + o(1))  \Bigl(\int_{\rho(x)}^{\ell(\rho(x))}\frac{g(2t)}{t}\,dt - (1+ o(1))\Bigl(-\ln \rho(x)\Bigr)^{1-\tau}\Bigr),
\end{align*}
where, by (\ref{eq:properties-ell}),
$$
\int_{\rho(x)}^{\ell(\rho(x))}\frac{g(2t)}{t}\,dt =\int_{\rho(x)}^{\ell(\rho(x))}\frac{1}{t\Bigl(-\ln t - \ln 2\Bigr)^\tau}\,dt\\
=-\frac{\bigl(-\ln t - \ln 2\bigr)^{1-\tau}}{1-\tau} \Big|_{\rho(x)}^{\ell(\rho(x))}=  \frac{1+ o(1)}{1-\tau}\Bigl(-\ln \rho(x)\Bigr)^{1-\tau}
$$
as $\rho(x) \to 0$. Consequently,
\begin{equation}
  \label{eq:suff-I-1-2}
I_{1,2}(x)\le \Bigl( \frac{\kappa_N \tau}{1-\tau}+ o(1)\Bigr)\Bigl(-\ln \rho(x)\Bigr)^{1-\tau}.
\end{equation}
To estimate $I_{1,1}(x)$, we use (\ref{eq:g-a-b-estimate}), (\ref{eq:sphere-est}) and Lemma~\ref{estimate-h-t} to find that
\begin{align}
I_{1,1}(x) &= \int_{\text{\tiny $|x| \!\le\! |y| \!\le \!|x|\!+\!\rho(x)$}}\frac{g(|y|-R)-g(|x|-R)}{|x-y|^{N} } dy \le
g'(|x|-R) \int_{\text{\tiny $|x|\! \le \!|y|\! \le \!|x|\!+\!\rho(x)$}}\frac{|y|-|x|}{|x-y|^{N} } dy \nonumber\\
&= g'(\rho(x)) \int_{|x|}^{|x|+\rho(x)}(t-|x|) t^{N-1} \int_{S^{N-1}}|x-t \theta|^{-N}\,d\theta d t=  g'(\rho(x)) \int_{|x|}^{|x|+\rho(x)}  h_0(|\frac{x}{t}|)d t \nonumber \\
&= (\kappa_N + o(1)) g'(\rho(x))\rho(x) = \frac{\kappa_N \tau + o(1)}{\bigl(-\log \rho(x)\bigr)^{\tau+1}} = o(1)\Bigl(-\ln \rho(x)\Bigr)^{1-\tau} \quad\ \text{as \,$\rho(x) \to 0$.}\label{suff-I-1-1-est}
\end{align}
Finally, we estimate $I_{1,3}(x)$ with the help of (\ref{eq:properties-ell}), (\ref{eq:sphere-est}) and Lemma~\ref{estimate-h-t}, finding that
\begin{align}
I_{1,3}(x) &\le 2 \|g\|_{L^\infty} \int_{|x|+ \ell(\rho(x))}^{|x|+1}t^{N-1} \int_{S^{N-1}} |x-ty|^{-N} dy = \int_{|x|+ \ell(\rho(x))}^{|x|+1}\frac{h_0(|\frac{x}{t}|)}{t-|x|}dt \nonumber\\
&\le 2 \|g\|_{L^\infty} \|h_0\|_{L^\infty}\Bigl(-\log \ell(\rho(x))\Bigr) = o(1)\Bigl(-\ln \rho(x)\Bigr)^{1-\tau} \quad\ \text{as \,$\rho(x) \to 0$.}\label{suff-I-1-3-est}
\end{align}
Combining (\ref{eq:suff-I-1-2}), (\ref{suff-I-1-1-est}) and (\ref{suff-I-1-3-est}), we get (\ref{eq:sufficient-I-1-cond}), and thus the claim follows.
\end{proof}

We now complete the proof of Theorem~\ref{pr 2.1-main}, which we restate here for the reader's convenience.

\begin{theorem}\label{pr 2.1}
Let $\Omega \subset \R^N$ be Lipschitz domain which satisfies a uniform exterior sphere condition,
 $f \in L^\infty(\Omega)$ and $u \in \cH(\Omega) \cap L^\infty(\R^N)$ be a weak solution of the Poisson problem
\begin{equation}
  \label{eq:poisson-f}
\loglap u = f \qquad \text{in $\Omega$}, \qquad u \equiv 0 \quad \text{in $\R^N \setminus \Omega.$}
\end{equation}
Then $u \in C(\overline \Omega)$ and
\begin{equation}\label{boundary-decay}
|u(x)|= O\bigl(\log^{-\tau} \frac1{\rho(x)}\bigr) \qquad \text{for every $\tau \in (0,\frac{1}{2})$ as $\rho(x) \to 0$.}
\end{equation}
\end{theorem}

\begin{proof}
Let
$$
f_1:= f + \j * u - \rho_N u \  \in L^\infty_{loc}(\R^N),
$$
where $\j$ is defined in (\ref{eq:def-j}). Then (\ref{eq:poisson-f}) is -- in weak sense -- equivalent to
$$
I u = f_1 \quad \text{in \,$\Omega$}, \qquad u \equiv 0 \quad \text{in \,$\R^N \setminus \Omega,$}
$$
where $I$ is the integral operator associated with the kernel $\k$ defined in (\ref{eq:def-k}), which is given by
$[I u](x) = \int_{\R^N}(u(x)-u(y))\k(x-y)\,dy$. Consequently, we have $u \in C_{loc}(\Omega)$ by the regularity result \cite[Theorem 3]{KM} of Kassmann and Mimica and a straightforward approximation argument. It thus remains to prove the estimate~(\ref{boundary-decay}) for fixed $\tau \in (0,\frac{1}{2})$. Since $\Omega$ satisfies a uniform exterior sphere condition, there exists a radius $0<R_0< \frac{1}{2}$ such that for every point $x_* \in \partial \Omega$ there exists a ball $B^{x_*}$ of radius $R_0$ contained in $\R^N \setminus \overline{\Omega}$ and
tangent to $\partial \Omega$ at $x_*$. Let $c(x_*)$ denote the center of $B^{x_*}$. We now apply Lemma~\ref{sec:regul-bound-decay-lemma-radial} with the value $R: =\frac{R_0}{2}$. This yields a function $V \in L^1_0(\R^N)$ and $\delta>0$ such that the properties (i)-(iv) of Lemma~\ref{sec:regul-bound-decay-lemma-radial} are satisfied. By Lemma~\ref{sec:regul-bound-decay-lemma-radial}(iv), we may assume that $\delta>0$ is chosen sufficiently small so that
$$
\loglap V(x) \ge \|f\|_{L^\infty} \quad \text{for $x \in A_{R,\delta}:= \{x \in \R^N \:: R<|x| <R+\delta\}$.}
$$
It now suffices to show that
\begin{equation}\label{boundary-estimate-omega-delta}
|u(x)|\le  c \bigl(- \log \rho(x)\bigr)^{-\tau} \qquad \text{for all $x \in \Omega_{\delta}$ with some constant $c>0$},
\end{equation}
where $\Omega_{\delta}=\{x\in\Omega:\, \rho(x)<\delta\}$.
For this we consider, for $t \in [0,1]$ and $x_* \in \partial \Omega$, the ball $B_R(z(t,x_*)) \subset B^{x_*}$ of radius $R$ centered at $z(t,x_*):= x_* + (t+R)\frac{c(x_*)-x_*}{|c(x_*)-x_*|}$. Moreover, we define the translated functions
$$
V_{t,x_*} \in L^1_0(\R^N), \qquad V_{t,x_*}(x)= V(x-z(t,x^*)) \qquad \text{for $x_* \in \partial \Omega$, $t \in [0,1]$.}
$$
For $t \in (0,\delta)$, the intersection $\Omega_{t,x_*}$ of $\Omega$ with the translated annulus $z(t,x_*)+A_{R,\delta}$ is nonempty. Moreover, making $\delta>0$ smaller if necessary, we may assume that the measure of $\Omega_{t,x_*}$ is small enough so that $\loglap$ satisfies the maximum principle on $\Omega_{t,x_*}$ (for all $x_* \in \partial \Omega$). Moreover, since $\Omega$ is bounded, there exists $R_1>R$ such that
$$
\Omega \subset \{x\in \R^N \:: R < |x- z(t,x_*)| < R_1\} \qquad \text{for all $x_* \in \partial \Omega$, $t \in (0,\delta)$},
$$
which implies that
$$
\Omega \setminus \Omega_{t,x_*} \subset \{x\in \R^N \:: R+\delta < |x- z(t,x_*)| < R_1\} \qquad \text{for all $x_* \in \partial \Omega$, $t \in (0,\delta)$.}
$$
Hence, since $V$ is positive on $\{x \in \R^N \:: R+\delta <|x| <R_1\}$ by Lemma~\ref{sec:regul-bound-decay-lemma-radial}(i), we may choose $c>1$ sufficiently large such that
$$
cV_{t,x_*} \ge \|u\|_{L^\infty} \qquad \text{in $\Omega \setminus \Omega_{t,x_*}$ for all $x_* \in \partial \Omega$, $t \in (0,\delta)$.}
$$
Consequently, we have, in weak sense,
$$
\loglap \bigl(c V_{t,x_*} \pm u\bigr) \ge 0 \quad \text{in $\Omega_{t,x_*}$,}\qquad c V_{t,x_*} \pm u \ge 0 \quad \text{on $\R^N \setminus \Omega_{t,x_*}$.}
$$
Here we note that $V_{t,x_*} \in \cV(\Omega)$ for $x_* \in \partial \Omega$, $t \in (0,\delta)$ since this function is uniformly Dini continuous on $\overline \Omega$ (for this we need $t>0$!). Applying the weak maximum principle to the function
$c V_{t,x*}\pm u$ on the set $\Omega_{t,x_*}$, we see that
$$
\pm u \le c V_{t,x_*} \qquad \text{on $\R^N$} \qquad \text{for every $x_* \in \partial \Omega$, $t \in (0,\delta)$.}$$
With $x_*$ being fixed, we may then pass to the limit $t \to 0$ and deduce that
$$
\pm u \le c V_{0,x_*} \qquad \text{on $\R^N$} \qquad \text{for every $x_* \in \partial \Omega$.}
$$
Next, we let $x \in \Omega$ with $\rho(x)< \delta$, and we let $x_* \in \partial \Omega$ with $\rho(x)= |x-x_*|$. It then also follows that
$$
0<\dist(x,B_R(z(0,x_*)))=\rho(x)<\delta
$$
and therefore, by Lemma~\ref{sec:regul-bound-decay-lemma-radial},
$$
x \in \Omega_{0,x*} \qquad \text{and}\qquad \pm u(x) \le c V_{0,x_*}(x)= c \bigl(-\ln \rho(x)\bigr)^{-\tau}.
$$
Hence~\eqref{boundary-estimate-omega-delta} is true, and the proof is finished.
\end{proof}

\begin{remark}
\label{remark-est-killing-measure}
{\rm Let $\Omega \subset \R^N$ be a domain, and let $\kappa_\Omega$ be the killing measure associated with the kernel $\k$ defined in (\ref{eq:def-kappa-omega}).
If $\Omega$ has a uniform $C^2$-boundary, then we may show by similar arguments as in the proofs of Lemma~\ref{estimate-kappa-Omega-ball}, Lemma~\ref{sec:regul-bound-decay-lemma-radial} and Theorem~\ref{pr 2.1} that 
  \begin{equation*}
\kappa_\Omega(x) = (c_N \kappa_N + o(1)) \log \frac{1}{\rho(x)}  \quad\ \text{as \,$\rho(x) \to 0$,}
  \end{equation*}
where $\kappa_N$ is given in (\ref{eq:def-kappa-N}).
}
\end{remark}

\section{Appendix: A logarithmic boundary Hardy inequality}
\label{sec:appendix}

The aim of this section is to prove the following logarithmic boundary Hardy inequality. It is not used in the present paper but might be of independent interest as it is related to the function spaces introduced in the previous sections, see Remark~\ref{a-posteriori-remark} above. Here, as before, we use the notation $\rho(x)= \dist(x,\partial \Omega)$ if the ambient domain $\Omega \subset \R^N$ is fixed.   

\begin{proposition}
\label{log-boundary-harnack-lip-bound}
Let $\Omega \subset \R^N$ be a bounded Lipschitz domain. Then there exists a constant $C=C(\Omega)>0$ with 
$$
\int_{\Omega} u^2(x) \log \frac{1}{\rho(x)} \,dx \le C\Bigl( {\mathfrak b}(u,\Omega) + \|u\|_{L^2(\Omega)}^2\Bigr)\qquad \text{for all $u \in C^\infty_c(\Omega)$,}  
$$
where, as before, 
$$
{\mathfrak b}(u,\Omega)= \frac{1} 2 \int_{\Omega} \int_{\Omega}(u(x)-u(y))^2\k(x-y) dx dy = \frac{1} 2 \int \!\!\!\!\int_{\stackrel{x,y \in \Omega}{\text{\tiny $|x-y|\!\le\! 1$}}} \frac{(u(x)-u(y))^2}{|x-y|^N} dx dy.
$$

\end{proposition}

\begin{corollary}
\label{log-boundary-harnack-lip-bound-corollary}
Let $\Omega \subset \R^N$ be a bounded Lipschitz domain. Then there exists a constant $C=C(\Omega)>0$ with 
$$\int_{\Omega} \kappa_\Omega(x) u^2(x)  \,dx \le C\Bigl( {\mathfrak b}(u,\Omega) + \|u\|_{L^2(\Omega)}^2\Bigr)\qquad \text{for all $u \in C^\infty_c(\Omega)$,}    
$$
where $\kappa_\Omega$ denotes the Killing measure corresponding to the kernel $\k$ defined in (\ref{eq:def-kappa-omega}).
\end{corollary}

\begin{proof}
This simply follows from Proposition~\ref{log-boundary-harnack-lip-bound} by noting that, since $\Omega$ is a bounded Lipschitz domain, we have that 
$\kappa_\Omega(x) \le 2 \log \frac{1}{\rho(x)} +C$ for $x \in \Omega$ with a constant $C>0$ by (\ref{eq:upper-estimate-kappa-omega}).
\end{proof}

The remainder of this section is devoted to the proof of Proposition~\ref{log-boundary-harnack-lip-bound}. The starting point of the proof is the following variant of Beckner's logarithmic inequality~(\ref{eq:beckner-inequality}). 

\begin{proposition}
\label{hyperplane-hardy}
We have   
$$
\int_{\R^{N}} \Bigl[\psi(\frac{1}{4})- \log (\pi |x_N|)\Bigr]u^2(x) dx 
\le \int_{\R^N}\log |\xi|  \, |\widehat{u}(\xi)|^2 d\xi\qquad \text{for Schwarz functions $u \in \cS(\R^N)$.} 
$$
\end{proposition}

\begin{proof}
We write $x=(z,t),\: \xi = (\tau, \zeta) \in \R^N$ with $z,\tau \in \R^{N-1}$, $t,\zeta \in \R$, and we recall the one-dimensional version of Beckner's inequality~(\ref{eq:beckner-inequality}):
\begin{equation}
  \label{eq:beckner-inequality-one-dim}
\int_{\R} \Bigl[\psi(\frac{1}{4})- \log (\pi |t|)\Bigr]u^2(t) dt \le \int_{\R} |\cF_1(u) (\zeta)|^2  \log |\zeta| d\zeta \qquad \text{for $u \in \cS(\R)$.} 
\end{equation}
Here $\cF_1$ denotes the one-dimensional Fourier transform. Next, let $u \in \cS(\R^N)$. Then (\ref{eq:beckner-inequality-one-dim}) implies that 
\begin{align*}
&\int_{\R^N} \Bigl[\psi(\frac{1}{4})- \log (\pi |x_N|)\Bigr]u^2(x) dt
= \int_{\R^{N-1}}\int_{\R} \Bigl[\psi(\frac{1}{4})- \log (\pi |t|)\Bigr]u^2(z,t) dt dz \\
&\le \int_{\R^{N-1}} \int_{\R}   |\cF_1(u(z,\cdot))(\zeta)|^2 \log |\zeta|d\zeta dz   = \int_{\R} \log |\zeta| \int_{\R^{N-1}} |\cF_1(u(z,\cdot))(\zeta)|^2 dz  d\zeta,\end{align*}
where, by the Plancherel theorem in $\R^{N-1}$,  
$$
\int_{\R^{N-1}} |\cF_1(u(z,\cdot))(\zeta)|^2 dz = \int_{\R^{N-1}}|\hat u (\tau,\zeta)|^2d\tau \qquad \text{for $\zeta \in \R$.}
$$
Using the fact that $\log$ is an increasing function, we conclude that 
$$
\int_{\R^{N}} \Bigl[\psi(\frac{1}{4})- \log (\pi |x_N|)\Bigr]u^2(x) dx 
\le \int_{\R} \log |\zeta| \int_{\R^{N-1}}|\hat u (\tau,\zeta)|^2d\tau d\zeta
\le \int_{\R^N}\log |\xi|\, |\hat u(\xi)|^2 d \xi.
$$
\end{proof}

In the remainder of this section, the latter $C$ always stands for a positive constant, and the value of $C$ may change in every step.

\begin{corollary}\label{corol:logar-hardy-ineq}   
There exists a constant $C>0$ such that 
$$
\int_{\R^N_+} u^2(x) \log \frac{1}{x_N} \,dx \le C \Bigl ({\mathfrak b}(u,\R^N_+) + \|u\|_{L^2(\R^N_+)}^2
 \Bigr)\quad \text{for all $u \in C^\infty_c(\R^N_+)$.}   
$$
\end{corollary}

\begin{proof}
Without loss of generality, let $u \in C^\infty_c(\R^N_+)$ be nonnegative, and let $v$ be the even extension of $u$ on $\R^N$ with respect to the reflection $\sigma: \R^N \to \R^N$ at the hyperplane $\{x_N= 0\}$. By Proposition~\ref{hyperplane-hardy}, 
\begin{align*}
&2 \int_{\R^{N}} \Bigl[\psi(\frac{1}{4})- \log (\pi |x_N|)\Bigr]v^2(x) dx 
\le 2\int_{\R^N}\log |\xi| |\widehat{v}(\xi)|^2 d\xi = \cE_L(v,v)\\
&= \cE(v,v)-c_N \int_{\R^N}\int_{\R^N} v(x)v(y)\j(x-y)\,dxdy + \rho_N \|v\|_{L^2(\R^N)}^2. 
\end{align*}
We thus infer that    
$$
\int_{\R^{N}} v^2(x) \log \frac{1}{|x_N|} dx \le C \Big(\cE(v,v) + \|v\|_{L^2(\R^N)}\Bigr).  
$$
Moreover, 
\begin{align*}
\cE(v,v)&= {\mathfrak b}(v,\R^N_+) + {\mathfrak b}(v,\R^N_-) + \int_{\R^N_+}\int_{\R^N \setminus \R^N_+}(v(x)-v(y))^2\k(x-y)\,dxdy\\   
&= 2 {\mathfrak b}(u,\R^N_+) + \int_{\R^N_+}\int_{\R^N_+}(u(x)-u(y))^2\k(x-\sigma(y))\,dxdy \le 4 {\mathfrak b}(u,\R^N_+). 
\end{align*}
Here we used in the last step that $\k(x-\sigma(y))\le \k(x-y)$ for $x,y \in \R^N_+$. We thus conclude that 
\begin{align*}
\int_{\R^{N}_+} u^2(x) \log \frac{1}{|x_N|} dx =\frac{1}{2} \int_{\R^{N}} v^2(x) \log \frac{1}{|x_N|} dx 
&\le C \Big(\cE(v,v) + \|v\|_{L^2(\R^N)}\Bigr)\\
&\le C \Big({\mathfrak b}(u,\R^N_+)  + 
\|u\|_{L^2(\R^N_+)}\Bigr),  
\end{align*}
as claimed. 
\end{proof}
In the following, let $\Omega \subset \R^N$ be a bounded open set. 

\begin{lemma}
\label{lipschitz-multiplicative}  
Let $\phi \in C^{0,1}(\Omega)$. Then there exists a constant $C=C(\phi)>0$ with 
$$
{\mathfrak b}(\phi u, \Omega) \le C \Bigl( {\mathfrak b}(u,\Omega)+ \|u\|_{L^2(\Omega)}^2\Bigr) \qquad \text{for all $u \in C^\infty_c(\Omega)$.}
$$
\end{lemma}

\begin{proof}
Let $u \in C^\infty_c(\Omega)$. Since 
$$
\bigl( [\phi u](x)-[\phi u](y)\bigr)^2 \le 2 \Bigl( 
\|\phi\|_{C^{0,1}(\Omega)}u^2(x)(x-y)^2 + \|\phi\|_{L^\infty(\Omega)}^2 \bigl(u(x)-u(y)\bigr)^2\Bigr)
$$
for $x,y \in \Omega$, we have  
$$
{\mathfrak b}(\phi u, \Omega)\le 2 \Bigl( \|\phi\|_{C^{0,1}(\Omega)} 
\int_{\Omega} u^2(x) \int_{\R^N}  |z|^2 \k(z) dz dx +  \|\phi\|_{L^\infty(\Omega)}^2 {\mathfrak b}(u,\Omega)\Bigr)
\le C \Bigl( \|u\|^2_{L^2(\Omega)} +{\mathfrak b}(u,\Omega)\Bigr).
$$The proof is complete.
\end{proof}

\begin{lemma}
\label{boundary-point-hardy}
Let $\Omega \subset \R^N$ be a bounded Lipschitz domain.  Then there exist $r>0$ and $C>0$ such that for any $z \in \partial \Omega$,
$$
\int_{\Omega}\log \frac{1}{\rho(x)}\, \phi^2(x)\,dx \le C \Bigl( {\mathfrak b}(\phi,\Omega) + \|\phi\|_{L^2(\Omega)}^2\Bigr) \quad \text{for all $\phi \in C^\infty_c( B_r(z) \cap \Omega)$.}
$$
\end{lemma}

\begin{proof}  Let
$$
Q:= \{ x \in \R^N\::\: |x| <1\},\qquad Q_+:= Q \cap \R^N_+ \quad \text{and} \quad Q_0:= \{x \in Q\::\: x_N=0\}.
$$
Since $\Omega \subset \R^N$ is a bounded Lipschitz domain, there exists  $r>0$ such that  for fixed $z \in \partial \Omega$ there exists a bilipschitz map $T: \R^N\to \R^N$ with 
$$T(z)=0,\ \  T(B_{2r}(z)\cap \partial\Omega)=  Q_0  \ \ {\rm and}\ \   
T(B_{2r}(z)\cap \Omega)= Q_+. 
 $$
For simplicity, let $B:= B_r(z)$ and $\tilde B:=  B_{2r}(z)$ in the following. Since $\rho(x)= \dist(x, \partial \Omega)= \dist(x, \partial \Omega \cap B)$ for $x \in \tilde B$, we have
\begin{equation}
  \label{eq:boundary-point-hardy-1}
(Tx)_N \le C \rho(x) \quad \text{and therefore}\quad \log \frac{1}{\rho(x)} \le \log \frac{1}{(T x)_N} + \log C \qquad \text{for $x \in \tilde B \cap \Omega$.}
\end{equation}
Let $\phi \in C^\infty_c(B \cap \Omega)$, and put $\psi:= \phi \circ T^{-1} \in C^{1,0}_c(Q_+)$. Then (\ref{eq:boundary-point-hardy-1}) and Corollary~\ref{corol:logar-hardy-ineq} imply that 
\begin{equation}
  \label{eq:boundary-point-hardy-2}  
\int_{\Omega}\log \frac{1}{\rho(x)} \phi^2(x)\,dx \le 
C \int_{Q_+} \Bigl(\log \frac{1}{y_N}+1\Bigr) \psi^2(y)\,dy \le 
C \Bigl ( \|\psi\|_{L^2(Q_+)}^2 +  
{\mathfrak b}(\psi,\R^N_+) \Bigr). 
\end{equation}
Moreover, since $\supp \psi \subset T^{-1}(\tilde B) \subset \subset Q_+$, we have 
\begin{equation}
  \label{eq:boundary-point-hardy-3}  
{\mathfrak b}(\psi,\R^N_+) \le C \Bigl( {\mathfrak b}(\psi,Q_+) + \|\psi\|_{L^2(Q_+)}^2\Bigr).
\end{equation}
Next we note that $\k (x-y) \le  C \bigl( \k(T^{-1}x-T^{-1}y) +1\bigr)$ for $x,y \in Q_+$ and therefore 
\begin{align}
{\mathfrak b}(\psi,Q_+) =\int_{Q_+}\int_{Q_+}(\psi(x)-\psi(y))^2 \k (x-y)dxdy &\le C \int_{\Omega} \int_{\Omega} (\phi(x)-\phi(y))^2 \Bigl( 
\k (x-y) +1\Bigr)dx dy \nonumber \\
&\le C \Bigl( {\mathfrak b}(\phi,\Omega) + \|\phi\|_{L^2(\Omega)}^2\Bigr).  \label{eq:boundary-point-hardy-4}  
\end{align}
The claim now follows by combining (\ref{eq:boundary-point-hardy-2}), (\ref{eq:boundary-point-hardy-3}) and (\ref{eq:boundary-point-hardy-4}).
\end{proof}

We may now complete the 

\begin{proof}[Proof of Proposition~\ref{log-boundary-harnack-lip-bound}] 
Let $r>0$ be chosen as in Lemma~\ref{boundary-point-hardy}. Since $\partial \Omega$ is compact, there exist points $x_1, \dots, x_n$ such that $\partial \Omega  \subset \bigcup \limits_{j=1}^n U_j$ with $U_j:= \tilde B_r(x_j)$. Setting $U_0:= \Omega$ gives rise to an open covering $\overline \Omega  \subset \bigcup \limits_{j=0}^n U_j$. We consider a subordinated partition of unity given by functions $\phi_0,\dots,\phi_n \in C_c^\infty(\R^N)$, i.e., 
$$
0 \le \phi_j \le 1, \quad \supp \phi_j \subset U_j \qquad \text{for $j=0,\dots,n$}\qquad \text{and}\qquad \text{$\sum \limits_{j=0}^n \phi_j  \equiv 1$ on $\overline \Omega$.}
$$
Now let $u \in C_c^\infty(\Omega)$. By Lemma~\ref{lipschitz-multiplicative} and Lemma~\ref{boundary-point-hardy}, 
$$
\int_{\Omega}\log \frac{1}{\rho(x)} (u \phi_j)^2(x)\,dx \le C \Bigl( {\mathfrak b}(u \phi_j,\Omega) + \|u \phi_j \|_{L^2(\Omega)}^2\Bigr)\le C \Bigl( {\mathfrak b}(u,\Omega) + \|u \|_{L^2(\Omega)}^2\Bigr)\quad \text{for $j= 1,\dots,n$.}
$$
Moreover, since $\supp \phi_0 \subset \subset \Omega$,  we have that 
$$
\int_{\Omega}\log \frac{1}{\rho(x)} (u \phi_0)^2(x)\,dx \le 
C \|u \phi_0 \|_{L^2(\Omega)}^2 \le C \|u\|_{L^2(\Omega)}^2.
$$
Combining these inequalities, we conclude that 
$$
\int_{\Omega}\log \frac{1}{\rho(x)}u^2(x)\,dx \le C 
\sum_{j=0}^n \int_{\Omega}\log \frac{1}{\rho(x)}(\phi_j u)^2(x)\,dx \le C \Bigl( {\mathfrak b}(u,\Omega) + \|u \|_{L^2(\Omega)}^2\Bigr), 
$$
as claimed.
\end{proof}

\noindent{\bf Acknowledgements:} We first wish to thank the referee for his/her helpful comments and suggestions. The second author also wishes to thank Mouhamed Moustapha Fall for valuable discussions. Part of this work was done while the first author was visiting the Goethe-Universit\"{a}t and he would like to
thank the mathematics institute for its hospitality.  H. Chen  is supported by NNSF of China, No: 11726614, 11661045,
by the Jiangxi Provincial Natural Science Foundation, No: 20161ACB20007 and by the Alexander von Humboldt Foundation.    T. Weth is supported by DAAD and BMBF (Germany) within the project 57385104.

\end{document}